\documentclass[a4paper]{amsart}
\usepackage[utf8]{inputenc}
\usepackage[T1]{fontenc}
\usepackage[english]{babel}

\usepackage{amsfonts}
\usepackage{amsmath}
\usepackage{amsrefs}
\usepackage{amssymb}
\usepackage{amstext}
\usepackage{amsthm}

\usepackage{geometry}
\usepackage{fullpage}

\usepackage[shortlabels]{enumitem}
\usepackage{float}
\usepackage{graphicx}
\usepackage{hyperref}
\usepackage{listings}
\usepackage{mathrsfs}
\usepackage{mathtools}
\usepackage{multicol}
\usepackage{shuffle}
\usepackage{wasysym}

\usepackage{tikz}
\usepackage{tikz-cd}
\usetikzlibrary{arrows, backgrounds, calc, chains, decorations, patterns, positioning, shapes}

\newcommand{\area}{\mathsf{area}}
\newcommand{\dinv}{\mathsf{dinv}}

\newcommand{\shift}{\mathsf{shift}}

\newcommand{\D}{\mathsf{D}} % Dyck paths (new)
\newcommand{\LD}{\mathsf{LD}} % Labelled Dyck paths
\newcommand{\SQ}{\mathsf{SQ}} % square paths
\newcommand{\LSQ}{\mathsf{LSQ}} % labelled square paths
\newcommand{\dw}{\mathsf{dw}}
\newcommand{\maj}{\mathsf{maj}}

\DeclareFontFamily{U}{bigshuffle}{}
\DeclareFontShape{U}{bigshuffle}{m}{n}{
	<5-8> s*[1.7] shuffle7
	<8->  s*[1.7] shuffle10
}{}
\DeclareSymbolFont{BigShuffle}{U}{bigshuffle}{m}{n}
\DeclareMathSymbol\bigshuffle{\mathop}{BigShuffle}{"001}
\DeclareMathSymbol\bigcshuffle{\mathop}{BigShuffle}{"002}

\renewcommand{\H}{\widetilde{H}}

\newcommand{\N}{\mathbb{N}}

\newcommand{\qbinom}[2]{\genfrac{[}{]}{0pt}{}{#1}{#2}}

\makeatletter

\pgfkeys{
	/tikz/sharp angle/.code={%
		\pgfsetarrowoptions{sharp >}{#1}%
		\pgfsetarrowoptions{sharp <}{-#1}%
	},
	/tikz/sharp > angle/.code={%
		\pgfsetarrowoptions{sharp >}{#1}%
	},
	/tikz/sharp < angle/.code={%
		\pgfsetarrowoptions{sharp <}{#1}%
	},
	/tikz/sharp protrude/.code=\csname if#1\endcsname\qrr@tikz@sharp@z@-0.05\p@\else\qrr@tikz@sharp@z@\z@\fi,
	/tikz/sharp protrude/.default=true
}

\newdimen\qrr@tikz@sharp@z@
\qrr@tikz@sharp@z@\z@
\pgfarrowsdeclare{sharp >}{sharp >}{%
	\edef\pgf@marshal{\noexpand\pgfutil@in@{and}{\pgfgetarrowoptions{sharp >}}}%
	\pgf@marshal
	\ifpgfutil@in@
	\edef\pgf@tempa{\pgfgetarrowoptions{sharp >}}
	\expandafter\qrr@tikz@sharp@parse\pgf@tempa\@qrr@tikz@sharp@parse
	\else
	\qrr@tikz@sharp@parse\pgfgetarrowoptions{sharp >}and-\pgfgetarrowoptions{sharp >}\@qrr@tikz@sharp@parse
	\fi
	\pgfmathparse{max(\pgf@tempa,\pgf@tempb,0)}%
	\let\qrr@tikz@sharp@max\pgfmathresult
	\pgfmathsetlength\pgf@xa{.5*\pgflinewidth * tan(\qrr@tikz@sharp@max)}%
	\pgfarrowsleftextend{+\pgf@xa}%
	\pgfarrowsrightextend{+\pgf@xa}%
}{%
	\edef\pgf@marshal{\noexpand\pgfutil@in@{and}{\pgfgetarrowoptions{sharp >}}}%
	\pgf@marshal
	\ifpgfutil@in@
	\edef\pgf@tempa{\pgfgetarrowoptions{sharp >}}
	\expandafter\qrr@tikz@sharp@parse\pgf@tempa\@qrr@tikz@sharp@parse
	\else
	\qrr@tikz@sharp@parse\pgfgetarrowoptions{sharp >}and-\pgfgetarrowoptions{sharp >}\@qrr@tikz@sharp@parse
	\fi
	\pgfmathsetlength\pgf@ya{.5*\pgflinewidth * tan(max(\pgf@tempa,\pgf@tempb,0))}%
	\pgfmathsetlength\pgf@xa{-.5*\pgflinewidth * tan(\pgf@tempa)}%
	\pgfmathsetlength\pgf@xb{-.5*\pgflinewidth * tan(\pgf@tempb)}%
	\advance\pgf@xa\pgf@ya
	\advance\pgf@xb\pgf@ya
	\ifdim\pgf@xa>\pgf@xb
	\pgftransformyscale{-1}%
	\pgf@xc\pgf@xb
	\pgf@xb\pgf@xa
	\pgf@xa\pgf@xc
	\fi
	\pgfpathmoveto{\pgfqpoint{\qrr@tikz@sharp@z@}{.5\pgflinewidth}}%
	\pgfpathlineto{\pgfqpoint{\pgf@xa}{.5\pgflinewidth}}%
	\pgfpathlineto{\pgfqpoint{\pgf@ya}{+0pt}}%
	\pgfpathlineto{\pgfqpoint{\pgf@xb}{-.5\pgflinewidth}}%
	\pgfpathlineto{\pgfqpoint{\qrr@tikz@sharp@z@}{-.5\pgflinewidth}}%
	\pgfusepathqfill
}
\pgfarrowsdeclare{sharp <}{sharp <}{%
	\edef\pgf@marshal{\noexpand\pgfutil@in@{and}{\pgfgetarrowoptions{sharp <}}}%
	\pgf@marshal
	\ifpgfutil@in@
	\edef\pgf@tempa{\pgfgetarrowoptions{sharp <}}
	\expandafter\qrr@tikz@sharp@parse\pgf@tempa\@qrr@tikz@sharp@parse
	\else
	\expandafter\qrr@tikz@sharp@parse\pgfgetarrowoptions{sharp <}and-\pgfgetarrowoptions{sharp <}\@qrr@tikz@sharp@parse
	\fi
	\pgfmathparse{max(\pgf@tempa,\pgf@tempb,0)}%
	\let\qrr@tikz@sharp@max\pgfmathresult
	\pgfmathsetlength\pgf@xa{.5*\pgflinewidth * tan(\qrr@tikz@sharp@max)}%
	\pgfarrowsleftextend{+\pgf@xa}%
	\pgfarrowsrightextend{+\pgf@xa}%
}{%
	\edef\pgf@marshal{\noexpand\pgfutil@in@{and}{\pgfgetarrowoptions{sharp <}}}%
	\pgf@marshal
	\ifpgfutil@in@
	\edef\pgf@tempa{\pgfgetarrowoptions{sharp <}}
	\expandafter\qrr@tikz@sharp@parse\pgf@tempa\@qrr@tikz@sharp@parse
	\else
	\expandafter\qrr@tikz@sharp@parse\pgfgetarrowoptions{sharp <}and-\pgfgetarrowoptions{sharp <}\@qrr@tikz@sharp@parse
	\fi
	\pgfmathsetlength\pgf@ya{.5*\pgflinewidth * tan(max(\pgf@tempa,\pgf@tempb,0))}%
	\pgfmathsetlength\pgf@xa{-.5*\pgflinewidth * tan(\pgf@tempa)}%
	\pgfmathsetlength\pgf@xb{-.5*\pgflinewidth * tan(\pgf@tempb)}%
	\advance\pgf@xa\pgf@ya
	\advance\pgf@xb\pgf@ya
	\ifdim\pgf@xa>\pgf@xb
	\pgftransformyscale{-1}%
	\pgf@xc\pgf@xb
	\pgf@xb\pgf@xa
	\pgf@xa\pgf@xc
	\fi
	\pgfpathmoveto{\pgfqpoint{\qrr@tikz@sharp@z@}{.5\pgflinewidth}}%
	\pgfpathlineto{\pgfqpoint{\pgf@xa}{.5\pgflinewidth}}%
	\pgfpathlineto{\pgfqpoint{\pgf@ya}{+0pt}}%
	\pgfpathlineto{\pgfqpoint{\pgf@xb}{-.5\pgflinewidth}}%
	\pgfpathlineto{\pgfqpoint{\qrr@tikz@sharp@z@}{-.5\pgflinewidth}}%
	\pgfusepathqfill
}
\def\qrr@tikz@sharp@parse#1and#2\@qrr@tikz@sharp@parse{\def\pgf@tempa{#1}\def\pgf@tempb{#2}}

\makeatother

\newcommand\multiset[2]%
{\mathchoice{\left(\kern-0.4em{\binom{#1}{#2}}\kern-0.4em\right)}
	{\bigl(\kern-0.2em{\binom{#1}{#2}}\kern-0.2em\bigr)}
	{\bigl(\kern-0.2em{\binom{#1}{#2}}\kern-0.2em\bigr)}
	{\bigl(\kern-0.2em{\binom{#1}{#2}}\kern-0.2em\bigr)}}

\let\existstemp\exists \renewcommand*{\exists}{\mathop \existstemp}
\let\foralltemp\forall \renewcommand*{\forall}{\mathop \foralltemp}

\def\quotient#1#2{\raise1ex\hbox{$#1$}\Big/\lower1ex\hbox{$#2$}}

\newcommand{\<}{\langle}
\renewcommand{\>}{\rangle}

\theoremstyle{plain}

\theoremstyle{definition}
\newtheorem{theorem}{Theorem}[section]
\newtheorem{conjecture}[theorem]{Conjecture}
\newtheorem{corollary}[theorem]{Corollary}
\newtheorem{definition}[theorem]{Definition}

\newtheorem{proposition}[theorem]{Proposition}

\theoremstyle{remark}
\newtheorem{remark}[theorem]{Remark}

\makeatletter%
\renewenvironment{proof}[1][\proofname]{%
	\par\pushQED{\qed}\normalfont%
	\topsep6\p@\@plus6\p@\relax
	\trivlist\item[\hskip\labelsep\bfseries#1\@addpunct{.}]%
	\ignorespaces
}{%
	\qedhere %\popQED\endtrivlist\@endpefalse
}
\makeatother

\makeatletter%
\DeclareRobustCommand*{\bfseries}{%
	\not@math@alphabet\bfseries\mathbf
	\fontseries\bfdefault\selectfont
	\boldmath
}
\makeatother

\title{A valley version of the Delta square conjecture}

\author{Alessandro Iraci}
\address{Universit\'e Libre de Bruxelles (ULB)\\D\'epartement de Math\'ematique\\ Boulevard du Triomphe, B-1050 Bruxelles\\ Belgium}\email{airaci@ulb.ac.be}

\author{Anna Vanden Wyngaerd}
\address{Universit\'e Libre de Bruxelles (ULB)\\D\'epartement de Math\'ematique\\ Boulevard du Triomphe, B-1050 Bruxelles\\ Belgium}\email{anvdwyng@ulb.ac.be}

\begin{document}
	
\begin{abstract}
Inspired by \cite{Qiu-Wilson-2019} and \cite{DAdderio-Iraci-VandenWyngaerd-DeltaSquare-2019}, we formulate a \emph{generalised Delta square conjecture} (valley version). Furthermore, we use similar techniques as in \cite{Haglund-Sergel-2019} to obtain a schedule formula for the combinatorics of our conjecture. We then use this formula to prove that the (generalised) valley version of the Delta conjecture implies our (generalised) valley version of the Delta square conjecture. This implication broadens the argument in \cite{Leven-2016}, relying on the formulation of the touching version in terms of the $\Theta_f$ operators introduced in \cite{DAdderio-Iraci-VandenWyngaerd-Theta-2019}.
\end{abstract}
	
\maketitle
\tableofcontents

\section{Introduction}
In \cite{Haglund-Remmel-Wilson-2018}, Haglund, Remmel and Wilson conjectured a combinatorial formula for $\Delta_{e_{n-k-1}}'e_n$ in terms of decorated labelled Dyck paths, which they called \emph{Delta conjecture}, after the so called delta operators $\Delta_f'$ introduced by Bergeron, Garsia, Haiman, and Tesler \cite{Bergeron-Garsia-Haiman-Tesler-Positivity-1999} for any symmetric function $f$. There are two versions of the conjecture, referred to as the \emph{rise} and the \emph{valley} version. 

In the same article \cite{Haglund-Remmel-Wilson-2018} the authors conjecture a combinatorial formula for the more general expression $\Delta_{h_m}\Delta_{e_{n-k-1}}'e_n$ in terms of decorated partially labelled Dyck paths, which we call \emph{generalised Delta conjecture} (rise version). In this paper, the authors also state a \emph{touching} refinement (where the number of times the Dyck path returns to the main diagonal is specified) of their conjecture. In \cite{DAdderio-Iraci-VandenWyngaerd-Theta-2019}, the authors introduce the $\Theta_f$ operators, and reformulate the touching version using these tools. In the present work, we will be using the latter formulation.

The Delta conjecture and its derivatives have attracted considerable attention since their formulation, see among others \cites{Wilson-Equidistribution, Rhoades-2018, Remmel-Wilson-2015, Zabrocki-Delta-Module-2019, Haglund-Rhoades-Shimozono-Advances, Garsia-Haglund-Remmel-Yoo-2019, DAdderio-Iraci-VandenWyngaerd-GenDeltaSchroeder-2019, DAdderio-Iraci-VandenWyngaerd-TheBible-2019, DAdderio-Iraci-VandenWyngaerd-DeltaSquare-2019, DAdderio-Iraci-VandenWyngaerd-Delta-t0-2018, Zabrocki-4Catalan-2016, Qiu-Wilson-2019,Haglund-Sergel-2019}. Most of the earlier work concerns the rise version, but interest in the valley version is growing. 

The special case $k=0$ of the Delta conjecture, which is known as the \emph{shuffle conjecture} \cite{HHLRU-2005}, was recently proved by Carlsson and Mellit \cite{Carlsson-Mellit-ShuffleConj-2018}. The shuffle theorem, thanks to the famous \emph{$n!$ conjecture}, now $n!$ theorem of Haiman \cite{Haiman-nfactorial-2001}, gives a combinatorial formula for the Frobenius characteristic of the $\mathfrak{S}_n$-module of diagonal harmonics studied by Garsia and Haiman.

In \cite{Loehr-Warrington-square-2007} Loehr and Warrington conjecture a combinatorial formula for $\Delta_{e_{n}}\omega(p_n)=\nabla \omega(p_n)$ in terms of labelled square paths (ending east), called the \emph{square conjecture}. The special case $\<\cdot ,e_n\>$ of this conjecture, known as the \emph{$q,t$-square}, has been proved by Can and Loehr in \cite{Can-Loehr-2006}. Recently the full square conjecture has been proved by Sergel in \cite{Leven-2016}, who showed that the shuffle theorem by Carlsson and Mellit \cite{Carlsson-Mellit-ShuffleConj-2018} implies the square conjecture (now square theorem).

In \cite{DAdderio-Iraci-VandenWyngaerd-DeltaSquare-2019} the authors conjecture a combinatorial formula for $\frac{[n-k]_t}{[n]_t}\Delta_{h_m}\Delta_{e_{n-k}}\omega(p_n)$ in terms of \emph{rise-decorated partially labelled square paths} that we call \emph{generalised Delta square conjecture} (rise version). This conjecture extends the square conjecture of Loehr and Warrington \cite{Loehr-Warrington-square-2007} (now a theorem \cite{Leven-2016}), i.e. it reduces to that one for $m=k=0$. Moreover, it extends the generalised Delta conjecture in the sense that on decorated partially labelled Dyck paths gives the same combinatorial statistics.

In \cite{Qiu-Wilson-2019}, the authors state a \emph{generalised Delta conjecture} (valley version), extending the valley version of the Delta conjecture. They also prove the case $q=0$, extending the results in \cite{DAdderio-Iraci-VandenWyngaerd-Delta-t0-2018}. 

Inspired by \cite{Qiu-Wilson-2019} and \cite{DAdderio-Iraci-VandenWyngaerd-DeltaSquare-2019}, we formulate two statements that can reasonably be called the \emph{generalised Delta square conjecture} (valley version). One is a combinatorial interpretation of the symmetric function \[\frac{[n-k]_q}{[n]_q}\Delta_{h_m}\Delta_{e_{n-k}}\omega(p_n)= \frac{[n]_t}{[n-k]_t} \Delta_{h_m} \Theta_{e_k} \nabla \omega(p_{n-k})\] (notice the swapping of $q$ and $t$ with respect to the rise version). The other is an interpretation of $\Delta_{h_m} \Theta_{e_k} \nabla \omega(p_{n-k})$, for which the combinatorics seems to be nicer and does not have the multiplicative factor.  

Next, we adapt the schedule formula in \cite{Haglund-Sergel-2019} to objects with repeated labels, which enabled us to incorporate the monomials into the formula. This allowed us to obtain a schedule formula for the combinatorics of our conjecture and to deal with the symmetric functions more easily. As a byproduct, our formula provides a new factorisation of all other previous schedule formulae concerning Dyck or square paths.

Finally, we use this formula to prove that the (generalised) valley version of the Delta conjecture implies our (generalised) valley version of the Delta square conjecture. This implication broadens the argument in \cite{Leven-2016}, relying on the formulation of the touching version in terms of the $\Theta_f$ operators. 
\section{Combinatorial definitions}
\begin{definition}
	A \emph{square path} of size $n$ is a lattice paths going from $(0,0)$ to $(n,n)$ consisting of east or north unit steps, always ending with an east step. The set of such paths is denoted by $\SQ(n)$. We call \emph{shift} of a square path the maximum value $s$ such that  the path intersect the line $y=x-s$ in at least one point. We refer to the line $y=x+i-s$ as \emph{$i$-th diagonal} and to the line $x=y$, (the $s$-th diagonal) as the \emph{main diagonal}. A \emph{Dyck path} is a square path whose shift is $0$. The set of Dyck paths is denoted by $\D(n)$. Of course $\D(n)\subseteq \SQ(n)$.  
\end{definition}

For example, the path in Figure~\ref{fig:labelled-square-path} has shift $3$. 

\begin{figure}[!ht]
	\centering
	\begin{tikzpicture}[scale = 0.6]
		\draw[step=1.0, gray!60, thin] (0,0) grid (8,8);
		\draw[gray!60, thin] (3,0) -- (8,5);
%		\draw[gray!60, thin] (8,6) -- (9,6) -- (9,8) (8,7) -- (10,7) -- (10,8) (8,8) -- (11,8);
%		\draw[gray!60, thin] (8,5) -- (11,8);
		
		\draw[blue!60, line width=1.6pt] (0,0) -- (0,1) -- (1,1) -- (2,1) -- (3,1) -- (4,1) -- (4,2) -- (5,2) -- (5,3) -- (5,4) -- (6,4) -- (6,5) -- (6,6) -- (6,7) -- (7,7) -- (7,8) -- (8,8);
		
%		\node at (5.5,5.5) {$\ast$};
		
		\node at (0.5,0.5) {$2$};
		\draw (0.5,0.5) circle (.4cm); 
		\node at (4.5,1.5) {$1$};
		\draw (4.5,1.5) circle (.4cm); 
		\node at (5.5,2.5) {$2$};
		\draw (5.5,2.5) circle (.4cm); 
		\node at (5.5,3.5) {$4$};
		\draw (5.5,3.5) circle (.4cm); 
		\node at (6.5,4.5) {$1$};
		\draw (6.5,4.5) circle (.4cm); 
		\node at (6.5,5.5) {$3$};
		\draw (6.5,5.5) circle (.4cm); 
		\node at (6.5,6.5) {$4$};
		\draw (6.5,6.5) circle (.4cm); 
		\node at (7.5,7.5) {$1$};
		\draw (7.5,7.5) circle (.4cm);
	\end{tikzpicture}
	\caption{Example of an element in $\LSQ(8)$.}
	\label{fig:labelled-square-path}
\end{figure}
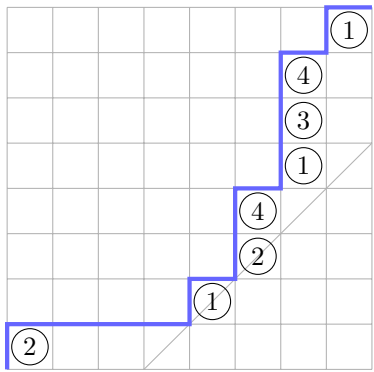

\begin{definition}
	Let $\pi$ be a square path of size $n$. We define its \emph{area word} to be the sequence of integers $a(\pi) = (a_1(\pi), a_2(\pi), \cdots, a_n(\pi))$ such that the $i$-th vertical step of the path starts from the diagonal $y=x+a_i(\pi)$. For example the path in Figure~\ref{fig:labelled-square-path} has area word $(0, \, -\!3, \, -\!3, \, -\!2, \, -\!2, \, -\!1, \, 0, \, 0)$.
\end{definition}

\begin{definition}
	A \emph{partial labelling} of a square path $\pi$ of size $n$ is an element $w \in \mathbb N^n$ such that
	\begin{itemize}
		\item if $a_i(\pi) > a_{i-1}(\pi)$, then $w_i > w_{i-1}$,
		\item $a_1(\pi) = 0 \implies w_1 > 0$,
		\item there exists an index $i$ such that $a_i(\pi) = - \shift(\pi)$ and $w_i(\pi) > 0$,
	\end{itemize}
	i.e. if we label the $i$-th vertical step of $\pi$ with $w_i$, then the labels appearing in each column of $\pi$ are strictly increasing from bottom to top, with the additional restrictions that, if the path starts north then the first label cannot be a $0$, and that there is at least one positive label lying on the base diagonal.
	
	We omit the word \emph{partial} if the labelling is composed of strictly positive labels only. %We will often, with an abuse of notation, identify a labelled square path with the underlying square path.
\end{definition}

\begin{definition}
	A \emph{(partially) labelled square path} (resp. \emph{Dyck path}) is a pair $(\pi, w)$ where $\pi$ is a square path (resp. Dyck path) and $w$ is a (partial) labelling of $\pi$. We denote by $\LSQ(m,n)$ (resp. $\LD(m,n)$) the set of labelled square paths (resp. Dyck paths) of size $m+n$ with exactly $n$ positive labels, and thus exactly $m$ labels equal to $0$.
\end{definition}

%The paths whose set of labels is exactly $[n]$, with $n$ being their size, are in some sense a (finite) set of representatives for the whole (infinite) set of labelled paths, so it will be useful to give a special name to them. We do so in terms of \emph{preference functions} and \emph{parking functions}.
%
%\begin{definition}
%	A \emph{preference function} is a function $f \colon [n] \rightarrow [n]$. A \emph{parking function} is a preference function such that  $\# \{ 1 \leq j \leq n \mid f(j) \geq i \} \leq n+1-i$.
%\end{definition}
%
%We denote by $\PR(n)$ (resp. $\PF(n)$) the set of preference (resp. parking) functions of size $n$. Given a square path whose set of labels is exactly $[n]$, we can determine a preference function by defining $f(j) = i$ if the label $j$ appears in the $i$-th column. It is easy to check that the correspondence is bijective, and that $f$ is a parking function if and only if it comes from a Dyck path. From now on, we will identify preference functions and parking functions with the corresponding labelled paths.

The following definitions will be useful later on.

\begin{definition}
	Let $w$ be a labelling of square path of size $n$. We define $x^w \coloneqq \prod_{i=1}^{n} x_{w_i} \rvert_{x_0 = 1}$.
\end{definition}

The fact that we set $x_0 = 1$ explains the use of the expression \emph{partially labelled}, as the labels equal to $0$ do not contribute to the monomial.

Sometimes we will, with an abuse of notation, write $\pi$ as a shorthand for a labelled path $(\pi, w)$. In that case, we use the identification $x^\pi \coloneqq x^w$.

Now we want to extend our sets introducing some decorations.

\begin{definition}
	\label{def:valley}
	The \emph{contractible valleys} of a labelled square path $\pi$ are the indices $1 \leq i \leq n$ such that one of the following holds:
	\begin{itemize}
		\item $i = 1$ and either $a_1(\pi) < -1$, or $a_1(\pi) = -1$ and $w_1 > 0$,
		\item $i > 1$ and $a_i(\pi) < a_{i-1}(\pi)$,
		\item $i > 1$ and $a_i(\pi) < a_{i-1}(\pi) \land w_i > w_{i-1}$.
	\end{itemize}
	
	We define \[ v(\pi, w) \coloneqq \{1 \leq i \leq n \mid i \text{ is a contractible valley} \}, \] corresponding to the set of vertical steps that are directly preceded by a horizontal step and, if we were to remove that horizontal step and move it after the vertical step, we would still get a square path with a valid labelling, with the additional restriction that if the vertical step is in the first row and it is attached to a $0$ label, then we require that it is preceded by at least two horizontal steps.
\end{definition}

\begin{remark}
	These slightly contrived conditions on the steps labelled $0$ have a more natural formulation in terms of steps labelled $\infty$, see Section~\ref{sec:concluding}
\end{remark}

This extends the definition of contractible valley given in \cite{Haglund-Remmel-Wilson-2018} to (partially) labelled square paths.

\begin{definition}
	\label{def:rise}
	The \emph{rises} of a (labelled) square path $\pi$ are the indices \[ r(\pi) \coloneqq \{2 \leq i \leq n \mid a_i(\pi) > a_{i-1}(\pi)\}, \] i.e. the vertical steps that are directly preceded by another vertical step. 
\end{definition}
	
\begin{definition}
	A \emph{valley-decorated (partially) labelled square path} is a triple $(\pi, w, dv)$ where $(\pi, w)$ is a (partially) labelled square path and $dv \subseteq v(\pi, w)$. A \emph{rise-decorated (partially) labelled square path} is a triple $(\pi, w, dr)$ where $(\pi, w)$ is a (partially) labelled square path and $dr \subseteq r(\pi)$.
\end{definition}

Again, we will often write $\pi$ as a shorthand for the corresponding triple $(\pi, w, dv)$ or $(\pi, w, dr)$.

We denote by $\LSQ(m,n)^{\bullet k}$ (resp. $\LSQ(m,n)^{\ast k}$) the set of partially labelled valley-decorated (resp. rise-decorated) square paths of size $m+n$ with $n$ positive labels and $k$ decorated contractible valleys (resp. decorated rises). We denote by $\LD(m,n)^{\bullet k}$ (resp. $\LD(m,n)^{\ast k}$) the corresponding subsets of Dyck paths. 

We also define $\LSQ'(m,n)^{\bullet k}$ as the set of paths in $\LSQ(m,n)^{\bullet k}$ such that there exists an index $i$ such that $a_i(\pi) = - \shift(\pi)$ and $i \not \in dv \land w_i(\pi) > 0$, i.e. there is at least one positive label lying on the bottom-most diagonal that is not a decorated valley. The importance of this set will be evident later in the paper.

Finally, we sometimes omit writing $m$ or $k$ when they are equal to $0$. Notice that, because of the restrictions we have on the labellings and the decorations, the only path with $n=0$ is the empty path, for which also $m=0$ and $k=0$.

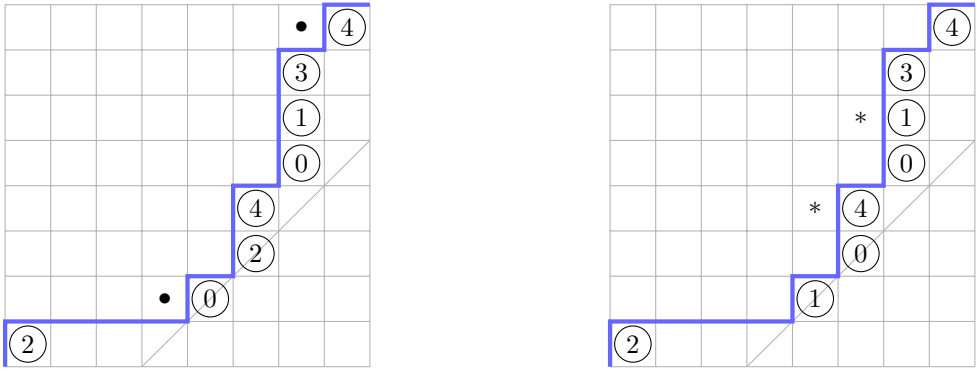
\begin{figure}[!ht]
	\begin{minipage}{0.5\textwidth}
		\centering
		\begin{tikzpicture}[scale = 0.6]
			\draw[step=1.0, gray!60, thin] (0,0) grid (8,8);
			\draw[gray!60, thin] (3,0) -- (8,5);
			%		\draw[gray!60, thin] (8,6) -- (9,6) -- (9,8) (8,7) -- (10,7) -- (10,8) (8,8) -- (11,8);
			%		\draw[gray!60, thin] (8,5) -- (11,8);
			
			\draw[blue!60, line width=1.6pt] (0,0) -- (0,1) -- (1,1) -- (2,1) -- (3,1) -- (4,1) -- (4,2) -- (5,2) -- (5,3) -- (5,4) -- (6,4) -- (6,5) -- (6,6) -- (6,7) -- (7,7) -- (7,8) -- (8,8);
			
			\node at (3.5,1.5) {$\bullet$};
			\node at (6.5,7.5) {$\bullet$};
			
			\node at (0.5,0.5) {$2$};
			\draw (0.5,0.5) circle (.4cm); 
			\node at (4.5,1.5) {$0$};
			\draw (4.5,1.5) circle (.4cm); 
			\node at (5.5,2.5) {$2$};
			\draw (5.5,2.5) circle (.4cm); 
			\node at (5.5,3.5) {$4$};
			\draw (5.5,3.5) circle (.4cm); 
			\node at (6.5,4.5) {$0$};
			\draw (6.5,4.5) circle (.4cm); 
			\node at (6.5,5.5) {$1$};
			\draw (6.5,5.5) circle (.4cm); 
			\node at (6.5,6.5) {$3$};
			\draw (6.5,6.5) circle (.4cm); 
			\node at (7.5,7.5) {$4$};
			\draw (7.5,7.5) circle (.4cm);
		\end{tikzpicture}
	\end{minipage}%
	\begin{minipage}{0.5\textwidth}
		\centering
		\begin{tikzpicture}[scale = 0.6]
			\draw[step=1.0, gray!60, thin] (0,0) grid (8,8);
			\draw[gray!60, thin] (3,0) -- (8,5);
			%		\draw[gray!60, thin] (8,6) -- (9,6) -- (9,8) (8,7) -- (10,7) -- (10,8) (8,8) -- (11,8);
			%		\draw[gray!60, thin] (8,5) -- (11,8);
			
			\draw[blue!60, line width=1.6pt] (0,0) -- (0,1) -- (1,1) -- (2,1) -- (3,1) -- (4,1) -- (4,2) -- (5,2) -- (5,3) -- (5,4) -- (6,4) -- (6,5) -- (6,6) -- (6,7) -- (7,7) -- (7,8) -- (8,8);
			
			\node at (4.5,3.5) {$\ast$};			
			\node at (5.5,5.5) {$\ast$};
			
			\node at (0.5,0.5) {$2$};
			\draw (0.5,0.5) circle (.4cm); 
			\node at (4.5,1.5) {$1$};
			\draw (4.5,1.5) circle (.4cm); 
			\node at (5.5,2.5) {$0$};
			\draw (5.5,2.5) circle (.4cm); 
			\node at (5.5,3.5) {$4$};
			\draw (5.5,3.5) circle (.4cm); 
			\node at (6.5,4.5) {$0$};
			\draw (6.5,4.5) circle (.4cm); 
			\node at (6.5,5.5) {$1$};
			\draw (6.5,5.5) circle (.4cm); 
			\node at (6.5,6.5) {$3$};
			\draw (6.5,6.5) circle (.4cm); 
			\node at (7.5,7.5) {$4$};
			\draw (7.5,7.5) circle (.4cm);
		\end{tikzpicture}
	\end{minipage}
	\caption{Example of an element in $\LSQ(2,6)^{\bullet 2}$ (left) and an element in $\LSQ(2,6)^{\ast 2}$ (right).}
	\label{fig:decorated-square-paths}
\end{figure}

We define two statistics on this set that reduce to the ones defined in \cite{Loehr-Warrington-square-2007} when $m=0$ and $k=0$. 

\begin{definition}
	\label{def:area}
	Let $(\pi, w, dr) \in \LSQ(m,n)^{\ast k}$ and $s$ be its shift. We define 
	\[ \area(\pi, w, dr) \coloneqq \sum_{i \not \in dr} (a_i(\pi) + s), \]
	i.e. the number of whole squares between the path and the base diagonal that are not in rows containing a decorated rise.
	
	For $(\pi, w, dv) \in \LSQ(m,n)^{\bullet k}$, we define $\area(\pi, w, dv) \coloneqq \area(\pi, w, \varnothing) \in \LSQ(m,n)^{\ast 0}$.
\end{definition}

For example, the path in Figure~\ref{fig:decorated-square-paths} (left) has area $8$. Notice that the area does not depend on the labelling.

\begin{definition}
	\label{def:dinv}
	Let $(\pi, w, dv) \in \LSQ(m,n)^{\bullet k}$. For $1 \leq i < j \leq n$, the pair $(i,j)$ is an \emph{inversion} if
	\begin{itemize}
		\item either $a_i(\pi) = a_j(\pi)$ and $w_i < w_j$ (\emph{primary inversion}),
		\item or $a_i(\pi) = a_j(\pi) + 1$ and $w_i > w_j$ (\emph{secondary inversion}),
%		\item or $i = j$ and $w_i >0$ and $a_i(\pi) < 0$ (\emph{bonus} or \emph{tertiary inversion}),
	\end{itemize}
	where $w_i$ denotes the $i$-th letter of $w$, i.e. the label of the vertical step in the $i$-th row. Then we define 
	\begin{align*}
		\dinv(\pi) \coloneqq \# \{ 1 \leq i < j \leq n \mid (i,j) \text{ inversion } \land j \not \in dv \} + \#\{1 \leq i \leq n \mid a_i(\pi) < 0 \land w_i > 0 \} - \# dv
	\end{align*}
	where again $\pi$ is a shorthand for $(\pi, w, dv)$.
	
	For $(\pi, w, dr) \in \LSQ(m,n)^{\ast k}$, we define $\dinv(\pi, w, dr) \coloneqq \dinv(\pi, w, \varnothing) \in \LSQ(m,n)^{\bullet 0}$.
\end{definition}

We refer to the middle term, counting the nonzero labels below the main diagonal, as \emph{bonus} or \emph{tertiary dinv}.

For example, the path in Figure~\ref{fig:decorated-square-paths} (right) has dinv equal to $3$: $1$ primary inversion in which the leftmost label is not a decorated valley, i.e. $(1,7)$; $1$ secondary inversion in which the leftmost label is not a decorated valley, i.e. $(1,6)$; $3$ bonus dinv, coming from the rows $3$, $4$, and $6$; $2$ decorated valleys. %Notice that Dyck paths coincide with the square paths with no bonus dinv.

It is easy to check that if $j \in dv$ then either there exists some inversion $(i,j)$ or $a_j < 0$. This means that if $m=0$ the dinv is always non-negative. In fact, thanks to the condition on the first row, the dinv is also non-negative for $m>0$, and also if $\pi \in \LSQ'(m,n)^{\bullet k}$ and it is not a Dyck path, then the dinv is necessarily strictly positive.

Finally, we recall two classical definitions.

\begin{definition}
	Let $p_1, \dots, p_k$ be a sequence of integers. We define its \emph{descent set} \[ \mathsf{Des}(p_1, \dots, p_k) \coloneqq \{ 1 \leq i \leq k-1 \mid p_i > p_{i+1} \} \] and its \emph{major index} $\maj(p_1, \dots, p_k)$ to be the sum of the elements of the descent set.
\end{definition}

\section{Symmetric functions} 
For all the undefined notations and the unproven identities, we refer to \cite{DAdderio-Iraci-VandenWyngaerd-TheBible-2019}*{Section~1}, where definitions, proofs and/or references can be found. 

We denote by $\Lambda$ the graded algebra of symmetric functions with coefficients in $\mathbb{Q}(q,t)$, and by $\<\, , \>$ the \emph{Hall scalar product} on $\Lambda$, defined by declaring that the Schur functions form an orthonormal basis.

The standard bases of the symmetric functions that will appear in our calculations are the monomial $\{m_\lambda\}_{\lambda}$, complete $\{h_{\lambda}\}_{\lambda}$, elementary $\{e_{\lambda}\}_{\lambda}$, power $\{p_{\lambda}\}_{\lambda}$ and Schur $\{s_{\lambda}\}_{\lambda}$ bases.

For a partition $\mu \vdash n$, we denote by \[ \H_\mu \coloneqq \H_\mu[X] = \H_\mu[X; q,t] = \sum_{\lambda \vdash n} \widetilde{K}_{\lambda \mu}(q,t) s_{\lambda} \] the \emph{(modified) Macdonald polynomials}, where \[ \widetilde{K}_{\lambda \mu} \coloneqq \widetilde{K}_{\lambda \mu}(q,t) = K_{\lambda \mu}(q,1/t) t^{n(\mu)} \] are the \emph{(modified) Kostka coefficients} (see \cite{Haglund-Book-2008}*{Chapter~2} for more details). 

Macdonald polynomials form a basis of the ring of symmetric functions $\Lambda$. This is a modification of the basis introduced by Macdonald \cite{Macdonald-Book-1995}.

If we identify the partition $\mu$ with its Ferrers diagram, i.e. with the collection of cells $\{(i,j)\mid 1\leq i\leq \mu_i, 1\leq j\leq \ell(\mu)\}$, then for each cell $c\in \mu$ we refer to the \emph{arm}, \emph{leg}, \emph{co-arm} and \emph{co-leg} (denoted respectively as $a_\mu(c), l_\mu(c), a_\mu(c)', l_\mu(c)'$) as the number of cells in $\mu$ that are strictly to the right, above, to the left and below $c$ in $\mu$, respectively.

Let $M \coloneqq (1-q)(1-t)$. For every partition $\mu$, we define the following constants:

\begin{align*}
	B_{\mu} & \coloneqq B_{\mu}(q,t) = \sum_{c \in \mu} q^{a_{\mu}'(c)} t^{l_{\mu}'(c)}, \\
	D_{\mu} & \coloneqq D_{\mu}(q,t) = MB_{\mu}(q,t)-1, \\
	T_{\mu} & \coloneqq T_{\mu}(q,t) = \prod_{c \in \mu} q^{a_{\mu}'(c)} t^{l_{\mu}'(c)} = q^{n(\mu')} t^{n(\mu)} = e_{\vert \mu \vert}[B_\mu], \\
	\Pi_{\mu} & \coloneqq \Pi_{\mu}(q,t) = \prod_{c \in \mu / (1,1)} (1-q^{a_{\mu}'(c)} t^{l_{\mu}'(c)}), \\
	w_{\mu} & \coloneqq w_{\mu}(q,t) = \prod_{c \in \mu} (q^{a_{\mu}(c)} - t^{l_{\mu}(c) + 1}) (t^{l_{\mu}(c)} - q^{a_{\mu}(c) + 1}).
\end{align*}

We will make extensive use of the \emph{plethystic notation} (cf. \cite{Haglund-Book-2008}*{Chapter~1}).

We need to introduce several linear operators on $\Lambda$.

\begin{definition}[\protect{\cite[3.11]{Bergeron-Garsia-ScienceFiction-1999}}]
	\label{def:nabla}
	We define the linear operator $\nabla \colon \Lambda \rightarrow \Lambda$ on the eigenbasis of Macdonald polynomials as \[ \nabla \H_\mu = T_\mu \H_\mu. \]
\end{definition}

\begin{definition}
	\label{def:pi}
	We define the linear operator $\mathbf{\Pi} \colon \Lambda \rightarrow \Lambda$ on the eigenbasis of Macdonald polynomials as \[ \mathbf{\Pi} \H_\mu = \Pi_\mu \H_\mu \] where we conventionally set $\Pi_{\varnothing} \coloneqq 1$.
\end{definition}

\begin{definition}
	\label{def:delta}
	For $f \in \Lambda$, we define the linear operators $\Delta_f, \Delta'_f \colon \Lambda \rightarrow \Lambda$ on the eigenbasis of Macdonald polynomials as \[ \Delta_f \H_\mu = f[B_\mu] \H_\mu, \qquad \qquad \Delta'_f \H_\mu = f[B_\mu-1] \H_\mu. \]
\end{definition}

Observe that on the vector space of symmetric functions homogeneous of degree $n$, denoted by $\Lambda^{(n)}$, the operator $\nabla$ equals $\Delta_{e_n}$. %Moreover, for every $k \in \mathbb{N}$, $\Delta_{e_k} = \Delta'_{e_k} + \Delta'_{e_{k-1}}$ on $\Lambda^{(n)}$, and since for any $k > n$, $\Delta_{e_k} = \Delta'_{e_{k-1}} = 0$ on $\Lambda^{(n)}$, then $\Delta_{e_n} = \Delta'_{e_{n-1}}$ on $\Lambda^{(n)}$.

We also introduce the Theta operators, first defined in \cite{DAdderio-Iraci-VandenWyngaerd-Theta-2019}

\begin{definition}
	\label{def:theta}
	For $f \in \Lambda$, we define the linear operator $\Theta_f \colon \Lambda \rightarrow \Lambda$ as \[ \Theta_f F[X] = \mathbf{\Pi} \; f \left[ \frac{X}{M} \right] \mathbf{\Pi}^{-1} F[X]. \]
\end{definition}

It is clear that $\Theta_f$ is linear, and moreover, if $f$ is homogenous of degree $k$, then so is $\Theta_f$, i.e. \[\Theta_f \Lambda^{(n)} \subseteq \Lambda^{(n+k)} \qquad \text{ for } f \in \Lambda^{(k)}. \]

%Finally, we define the Pieri coefficients as follows.
%
%\begin{definition}
%	\label{def:pieri-coefficients}
%	For $k \in \mathbb{N}$ and $f \in \Lambda^{(k)}$, we define the Pieri coefficients $c_{\mu \nu}^{f^\perp}, d_{\mu \nu}^{f}$ by
%	\begin{align*}
%	f[X]^\perp \H_\mu[X] & = \sum_{\nu \subset_k \mu} c_{\mu \nu}^{f^\perp} \H_\nu[X], \\
%	f[X] \H_\nu[X] & = \sum_{\mu \supset_k \nu} d_{\mu \nu}^{f} \H_\mu[X].
%	\end{align*}
%	where $\nu \subset_k \mu$ means that $\nu \subset \mu$ and $\lvert \mu \rvert - \lvert \nu \rvert = k$.
%\end{definition}
%
%We can immediately derive that \[ w_\nu c_{\mu \nu}^{f^\perp} = \left\< f^\perp \H_\mu[X], \H_\nu \right\>_* = \left\< \H_\mu[X], \omega f \left[ \frac{X}{M}\right] \H_\nu[X] \right\>_* = w_\mu d_{\mu \nu}^{\omega f[X/M]} \] so these two families of coefficients determine each other. It is convenient to define $c_{\mu \nu}^{(k)}$, $d_{\mu \nu}^{(k)}$ by \[ {h_k}^\perp \H_\mu[X] = \sum_{\nu \subset_k \mu} c_{\mu \nu}^{(k)} \H_\nu[X] \quad \text{and} \quad	{e_k}\left[\frac{X}{M}\right] \H_\nu[X] = \sum_{\mu \supset_k \nu} d_{\mu \nu}^{(k)} \H_\mu[X]. \]

It is convenient to introduce the so called $q$-notation. In general, a $q$-analogue of an expression is a generalisation involving a parameter $q$ that reduces to the original one for $q \rightarrow 1$.

\begin{definition}
	For a natural number $n \in \mathbb{N}$, we define its $q$-analogue as \[ [n]_q \coloneqq \frac{1-q^n}{1-q} = 1 + q + q^2 + \dots + q^{n-1}. \]
\end{definition}

Given this definition, one can define the $q$-factorial and the $q$-binomial as follows.

\begin{definition}
	We define \[ [n]_q! \coloneqq \prod_{k=1}^{n} [k]_q \quad \text{and} \quad \qbinom{n}{k}_q \coloneqq \frac{[n]_q!}{[k]_q![n-k]_q!} \]
\end{definition}

\begin{definition}
	For $x$ any variable and $n \in \N \cup \{ \infty \}$, we define the \emph{$q$-Pochhammer symbol} as \[ (x;q)_n \coloneqq \prod_{k=0}^{n-1} (1-xq^k) = (1-x) (1-xq) (1-xq^2) \cdots (1-xq^{n-1}). \]
\end{definition}

We can now introduce yet another family of symmetric functions.

\begin{definition}
	\label{def:Enk}
	For $0 \leq k \leq n$, we define the symmetric function $E_{n,k}$ by the expansion \[ e_n \left[ X \frac{1-z}{1-q} \right] = \sum_{k=0}^n \frac{(z;q)_k}{(q;q)_k} E_{n,k}. \]
\end{definition}

Notice that $E_{n,0} = \delta_{n,0}$. Setting $z=q^j$ we get \[ e_n \left[ X \frac{1-q^j}{1-q} \right] = \sum_{k=0}^n \frac{(q^j;q)_k}{(q;q)_k} E_{n,k} = \sum_{k=0}^n \qbinom{k+j-1}{k}_q E_{n,k} \] and in particular, for $j=1$, we get \[ e_n = E_{n,0} + E_{n,1} + E_{n,2} + \cdots + E_{n,n}, \] so these symmetric functions split $e_n$, in some sense.

We care in particular about the following identity.

\begin{proposition}[\protect{\cite[Theorem~4]{Can-Loehr-2006}}]
	\label{prop:pn_Enk}
	\[ \omega(p_n) = \sum_{k=1}^n \frac{[n]_q}{[k]_q} E_{n,k} \]
\end{proposition}

The Theta operators will be useful to restate the Delta conjectures in a new fashion, thanks to the following results.

\begin{theorem}[\protect{\cite[Theorem~3.1]{DAdderio-Iraci-VandenWyngaerd-Theta-2019}}]
	\label{thm:theta-en}
	\[ \Theta_{e_k} \nabla e_{n-k} = \Delta'_{e_{n-k-1}} e_n \]
\end{theorem}

\begin{theorem}[\protect{\cite[Theorem~3.3]{DAdderio-Iraci-VandenWyngaerd-Theta-2019}}]
	\label{thm:theta-pn}
	\[ \frac{[n]_q}{[n-k]_q} \Theta_{e_k} \nabla \omega(p_{n-k}) = \frac{[n-k]_t}{[n]_t} \Delta_{e_{n-k}} \omega(p_n) \]
\end{theorem}

\begin{corollary}
	\[ \frac{[n]_t}{[n-k]_t} \Theta_{e_k} \nabla \omega(p_{n-k}) = \frac{[n-k]_q}{[n]_q} \Delta_{e_{n-k}} \omega(p_n) \]
\end{corollary}

%\begin{proof}
%	Both $\Theta_{e_k} \nabla \omega(p_{n-k})$ and $\Delta_{e_{n-k}} \omega(p_n)$ are symmetric in $q,t$, so swapping the variables in Theorem~\ref{thm:theta-pn} yields the thesis.
%\end{proof}
\section{Delta conjectures}
By \emph{Delta conjectures} we refer to a family of conjectures that provide a combinatorial interpretation of certain symmetric functions that arise from the Delta operators and show positivity properties. %With an abuse of notation, from now on we will use the identifications $\pi = (\pi, w, dr)$ or $\pi = (\pi, w, dv)$ when dealing with labelled decorated square (or Dyck) paths.

The first and most famous of the Delta conjectures is known as \emph{shuffle conjecture}, now a theorem by E. Carlsson and A. Mellit (see \cite{Carlsson-Mellit-ShuffleConj-2018}).

\begin{theorem}[Shuffle theorem]
	\[ \nabla e_n = \sum_{\pi \in \LD(n)} q^{\dinv(\pi)} t^{\area(\pi)} x^\pi \]% = \sum_{\pi \in \PF(n)} q^{\dinv(\pi)} t^{\area(\pi)} Q_{\mathsf{ides}(\sigma(\pi))}. \]
\end{theorem}

%In \cite{Loehr-Remmel-2004} the authors describe a bijection of the set of parking functions (i.e. labelled Dyck paths of size $n$ whose set of labels is exactly $[n]$) with itself mapping the bistatistic $(\dinv, \area)$ to $(\area, \pmaj)$. As a corollary, the following holds.
%
%\begin{corollary}[Shuffle Theorem, parking version]
%	\[ \nabla e_n = \sum_{\pi \in \LD(n)} q^{\area(\pi)} t^{\pmaj(\pi)} x^\pi. \]% = \sum_{\pi \in \PF(n)} q^{\area(\pi)} t^{\pmaj(\pi)} Q_{\mathsf{ides}(\tau(\pi))}. \]
%\end{corollary}

The shuffle theorem is especially important because $\nabla e_n$ has another interpretation, as the bigraded Frobenius characteristic of the $S_n$ module of the \emph{diagonal harmonics}. This is one of the facts that first motivated the study of Macdonald polynomials, and it has been proved by M. Haiman in \cite{Haiman-Vanishing-2002}. See also \cite{Haiman-nfactorial-2001} for further details.

The \emph{Delta conjecture} is a generalisation of the shuffle conjecture, introduced by J. Haglund, J. Remmel, and A. Wilson in \cite{Haglund-Remmel-Wilson-2018}. In the same paper, the authors suggest that an even more general conjecture should hold, which we call \emph{generalised Delta conjecture}. It reads as follows.

\begin{conjecture}[(Generalised) Delta conjecture, valley version]
	\label{conj:valley-delta}
	\[ \Delta_{h_m} \Delta'_{e_{n-k-1}} e_n = \sum_{\pi \in \LD(m,n)^{\bullet k}} q^{\dinv(\pi)} t^{\area(\pi)} x^\pi. \]
\end{conjecture}

For $m=0$ this conjecture first appears together with the rise version in \cite{Haglund-Remmel-Wilson-2018}. The full statement, together with a proof of the case $q=0$, has been given by D. Qiu and A. Wilson in \cite{Qiu-Wilson-2019}.

\begin{conjecture}[(Generalised) Delta conjecture, rise version]
	\[ \Delta_{h_m} \Delta'_{e_{n-k-1}} e_n = \sum_{\pi \in \LD(m,n)^{\ast k}} q^{\dinv(\pi)} t^{\area(\pi)} x^\pi. \]
\end{conjecture}

The rise version of the Delta conjecture is simply the case $m=0$ of the general case.

% TODO Stuff on the two versions.

Recalling that $\nabla \rvert_{\Lambda^{(n)}} = \Delta'_{e_{n-1}} \rvert_{\Lambda^{(n)}}$, it is clear that for $k=0$ both the versions of the Delta conjecture reduce to the shuffle theorem.

%In analogy with the shuffle conjecture, the Delta conjecture also has a version in terms of $(\area, \pmaj)$. Unfortunately the bijection by N. Loehr and J. Remmel does not generalise for $k>0$, so in the general case it is not proved that the two conjectures are equivalent.
%
%\begin{conjecture}[(Generalised) Delta conjecture, parking version]
%	\[ \Delta_{h_m} \Delta'_{e_{n-k-1}} e_n = \sum_{\pi \in \LD(m,n)^{\ast k}} q^{\area(\pi)} t^{\pmaj(\pi)} x^\pi. \]
%\end{conjecture}
%
%where again the $\pmaj$ extends in the obvious way when $0$ labels are introduced. %We want to emphasise the $(\area, \pmaj)$ version, despite it being less popular in the literature, because certain symmetric functions are easier to interpret in this case.

The \emph{square conjecture} was first suggested by N. Loehr and G. Warrington in \cite{Loehr-Warrington-square-2007}, and it was then proved by E. Sergel in \cite{Leven-2016} using the shuffle theorem.

\begin{theorem}[Square Theorem]
	\[ \nabla \omega(p_n) = \sum_{\pi \in \LSQ(n)} q^{\dinv(\pi)} t^{\area(\pi)} x^\pi. \]% = \sum_{\pi \in \PR(n)} q^{\dinv(\pi)} t^{\area(\pi)} Q_{\mathsf{ides}(\sigma(\pi))}. \]
\end{theorem}

Unfortunately, adding zero labels and decorated rises to square paths in the trivial way and $q,t$-counting the resulting objects with respect to the bistatistic $(\dinv, \area)$ gives a polynomial that does not match the expected symmetric function. This issue has been addressed by M. D'Adderio and the authors, who stated the \emph{generalised Delta square conjecture} in \cite{DAdderio-Iraci-VandenWyngaerd-DeltaSquare-2019}.

\begin{conjecture}[(Generalised) Delta square conjecture, rise version]
	\[ \frac{[n-k]_t}{[n]_t} \Delta_{h_m} \Delta_{e_{n-k}} \omega(p_n) = \sum_{\pi \in \LSQ(m,n)^{\ast k}} q^{\dinv(\pi)} t^{\area(\pi)} x^\pi. \]
\end{conjecture}

where the rise version of the square conjecture is simply the case $m=0$ of the general case.

The square conjectures used to lack a valley version. Computational evidence suggests the following, checked by computer up to $n=6$.

\begin{conjecture}[(Generalised) Delta square conjecture, valley version]
	\label{conj:gen-valley-square}
	\[ \frac{[n-k]_q}{[n]_q} \Delta_{h_m} \Delta_{e_{n-k}} \omega(p_n) = \sum_{\pi \in \LSQ(m,n)^{\bullet k}} q^{\dinv(\pi)} t^{\area(\pi)} x^\pi. \]
\end{conjecture}

Notice that the symmetric function we propose for the valley version differs from the one appearing in the rise version (for $m=0$) as the multiplicative factor is the ratio of two $q$-analogues instead of two $t$-analogues. This suggests a potential extension of the conjecture to a version with both decorated rises and contractible valleys, possibly using the Theta operators appearing in \cite{DAdderio-Iraci-VandenWyngaerd-Theta-2019}. The power series associated to the obvious combinatorial extension, however, seems to be quasi-symmetric function which is not symmetric, and thus further investigation is required to find suitable statistics.

We can restate it in terms of Theta operators as follows.

\begin{conjecture}[(Generalised) Delta square conjecture, valley version]
	\label{conj:gen-valley-square-theta}
	\[ \frac{[n]_t}{[n-k]_t} \Delta_{h_m} \Theta_{e_k} \nabla \omega(p_{n-k}) = \sum_{\pi \in \LSQ(m,n)^{\bullet k}} q^{\dinv(\pi)} t^{\area(\pi)} x^\pi. \]
\end{conjecture}

We need to state a refinement of the Delta conjecture, valley version, that naturally arises when stating it in terms of Theta operators. But first, we need another combinatorial definition. Let \[ \LSQ(m, n \backslash r)^{\bullet k} \coloneqq \{ \pi \in \LSQ(n)^{\bullet k} \mid \# \{ i \not \in dv \colon a_i = - \shift(\pi) \land w_i > 0 \} = r \}, \] which is the set of labelled valley-decorated square paths of size $m+n$ with $m$ labels equal to $0$ and $k$ decorations such that there are exactly $r$ steps which are neither $0$ labels not decorated valleys on the bottom-most diagonal, and let \[ \LD(m, n \backslash r)^{\bullet k} \coloneqq \LSQ(m, n \backslash r)^{\bullet k} \cap \LD(m, n)^{\bullet k}, \]	which is the subset of corresponding labelled valley-decorated Dyck paths. We state the following.

\begin{conjecture}[Touching Delta conjecture, valley version]
	\label{conj:valley-delta-touching}
	\[ \Theta_{e_k} \nabla E_{n-k, r} = \sum_{\pi \in \LD(n \backslash r)^{\bullet k}} q^{\dinv(\pi)} t^{\area(\pi)} x^\pi. \]
\end{conjecture}

It is immediate that Conjecture~\ref{conj:valley-delta-touching} implies the case $m=0$ of Conjecture~\ref{conj:valley-delta}, as it is enough to sum over $r$ and then apply Theorem~\ref{thm:theta-en}.

We need to state the same refinement for the generalised version too.

\begin{conjecture}[Generalised touching Delta conjecture, valley version]
	\label{conj:gen-valley-delta-touching}
	\[ \Delta_{h_m} \Theta_{e_k} \nabla E_{n-k, r} = \sum_{\pi \in \LD(m, n \backslash r)^{\bullet k}} q^{\dinv(\pi)} t^{\area(\pi)} x^\pi. \]
\end{conjecture}

We now want to state yet another version of the Delta square conjecture, using the set $\LSQ'(m,n)^{\bullet k}$ previously introduced.

\begin{conjecture}[Modified Delta square conjecture, valley version]
	\label{conj:valley-square-2}
	\[ \Theta_{e_k} \nabla \omega(p_{n-k}) = \sum_{\pi \in \LSQ'(n)^{\bullet k}} q^{\dinv(\pi)} t^{\area(\pi)} x^\pi. \]
\end{conjecture}

This conjecture is new and it is more nice-looking than the other forms of the Delta square conjecture as it does not have any multiplicative correcting factor. It also extends nicely to the $m>0$ case, as follows.

\begin{conjecture}[Modified generalised Delta square conjecture, valley version]
	\label{conj:gen-valley-square-2}
	\[ \Delta_{h_m} \Theta_{e_k} \nabla \omega(p_{n-k}) = \sum_{\pi \in \LSQ'(m,n)^{\bullet k}} q^{\dinv(\pi)} t^{\area(\pi)} x^\pi. \]
\end{conjecture}

Our goal is to show that Conjecture~\ref{conj:gen-valley-delta-touching} implies Conjecture~\ref{conj:gen-valley-square-2}, and as a corollary that Conjecture~\ref{conj:valley-delta-touching} implies Conjecture~\ref{conj:valley-square-2}.
\section{Schedule numbers for repeated labels}
%It is convenient to introduce the set $\N^\bullet\coloneqq \{\stackrel{\bullet}{0}, \stackrel{\bullet}{1}, \dots\}$.

\begin{definition}
	Let $(\pi, w, dv)$ be a valley-decorated labelled square path with shift $s$. For $i \geq 0$ we set $\rho_i$ to be the marked word in the alphabet $\N$ consisting of the labels appearing in the $i$-th diagonal, marked with a $\bullet$ if it labels a decorated valley, in increasing order, where we consider $c <\;\stackrel{\bullet}{\raisebox{0 em}{c}}\;<c+1$. The \emph{diagonal word} of $(\pi, w, dv)$ is $\dw(\pi, w, dv) \coloneqq \rho_\ell \dots \rho_{1} \rho_{0}$.
\end{definition}

For example the diagonal word of the path in Figure~\ref{fig:diagonal-word} is $1243 \!\! \stackrel{\bullet}{\raisebox{0 em}{1}} \!\! 4 1 \!\! \stackrel{\bullet}{\raisebox{0 em}{1}}$. Notice that the $\rho_i$ are the runs of $\dw(\pi, w, dv)$, i.e. the maximal weakly increasing substrings (disregarding decorations). 

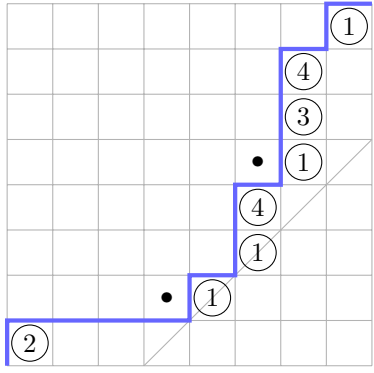
\begin{figure}[!ht]
	\begin{tikzpicture}[scale = 0.6]
			\draw[step=1.0, gray!60, thin] (0,0) grid (8,8);
			\draw[gray!60, thin] (3,0) -- (8,5);
			%		\draw[gray!60, thin] (8,6) -- (9,6) -- (9,8) (8,7) -- (10,7) -- (10,8) (8,8) -- (11,8);
			%		\draw[gray!60, thin] (8,5) -- (11,8);
			
			\draw[blue!60, line width=1.6pt] (0,0) -- (0,1) -- (1,1) -- (2,1) -- (3,1) -- (4,1) -- (4,2) -- (5,2) -- (5,3) -- (5,4) -- (6,4) -- (6,5) -- (6,6) -- (6,7) -- (7,7) -- (7,8) -- (8,8);
			
			\node at (3.5,1.5) {$\bullet$};
			\node at (5.5,4.5) {$\bullet$};
			
			\node at (0.5,0.5) {$2$};
			\draw (0.5,0.5) circle (.4cm); 
			\node at (4.5,1.5) {$1$};
			\draw (4.5,1.5) circle (.4cm); 
			\node at (5.5,2.5) {$1$};
			\draw (5.5,2.5) circle (.4cm); 
			\node at (5.5,3.5) {$4$};
			\draw (5.5,3.5) circle (.4cm); 
			\node at (6.5,4.5) {$1$};
			\draw (6.5,4.5) circle (.4cm); 
			\node at (6.5,5.5) {$3$};
			\draw (6.5,5.5) circle (.4cm); 
			\node at (6.5,6.5) {$4$};
			\draw (6.5,6.5) circle (.4cm); 
			\node at (7.5,7.5) {$1$};
			\draw (7.5,7.5) circle (.4cm);
		\end{tikzpicture}
	\caption{Square path of diagonal word $1243 \!\! \stackrel{\bullet}{\raisebox{0 em}{1}} \!\! 4 1 \!\! \stackrel{\bullet}{\raisebox{0 em}{1}}$} 
	\label{fig:diagonal-word}
\end{figure}

\begin{definition}
	Let $z \coloneqq \dw(\pi, w, dv)$ be the diagonal word of a valley-decorated labelled square path $(\pi, w, dv)$ such that $z = \rho_\ell \cdots \rho_0$, where the $\rho_i$'s are its runs. We define its \emph{$i$-th run multiplicity functions}
	\begin{align*}
		z_i \colon \N & \rightarrow \N \\
		 x & \mapsto \# \{ x \in \rho_i \mid x \text{  is undecorated}\} \\
		z_i^\bullet \colon \N & \rightarrow \N \\
		 x & \mapsto \# \{ x \in \rho_i \mid x \text{  is decorated}\}
	\end{align*}
	Notice that each function $z_i$ has finite support. 
\end{definition}

\begin{definition}
	Given $z=\rho_\ell, \dots, \rho_0$ the diagonal word of a square path $(\pi,w)$ with shift $s$, where the $\rho_i$ are the runs of $z$. We set $\tilde \rho_i$ to be the subword obtained from $\rho_i$ by deleting its decorated numbers and $\rho'_i$ the subword obtained from $\tilde \rho_i$ by deleting the zeros. Take $c \in z$, for $i \in \{0,\dots, \ell\}$, we define its \emph{schedule numbers} $w_{i,s}(c)$ as follows: For $c \in \N$
		\begin{align*}
			w_{i,s}(c) &\coloneqq 
			\begin{cases}
				\#\{d \in \tilde\rho_i \mid d>c \} + \#\{d \in \tilde\rho_{i-1} \mid d<c \}  & \text{if  } i\in \{s+1,\dots, \ell\} \\
				\#\{d \in \tilde\rho_i \mid d>c \} + 1 - \delta_{c,0} & \text{if  } i=s \\
				\#\{d \in \tilde\rho_i \mid d<c \} + \#\{d \in \tilde\rho_{i+1} \mid d>c \}  & \text{if  } i\in \{0,\dots, s-1\}
			\end{cases}\\
			w_{i,s}^\bullet (c) &\coloneqq \#\{d \in \tilde\rho_i \mid d<c \}+\#\{d \in \tilde\rho_{i+1} \mid d>c \} - \delta_{c,0}\delta_{i,s-1} 
		\end{align*}

Careful! These cardinalities of multisets take into account the multiplicities. 
% TODO: find a better way to write this.
\end{definition}

%\begin{remark}
%	It can be shown that $w_{i,s}(c)>0$ for all $i,s,c$.
%\end{remark}

\begin{theorem}
	\label{thm:factorisation}
	Let $z$ be a marked word in the alphabet $\N$ and let $\rho_\ell, \dots, \rho_0$ be the runs of $z$, so that $z = \rho_\ell \cdots \rho_0$. Let $b(z,s) \coloneqq \sum_{i=0}^{s-1} \sum_{c \in \N} z_i(c) - z_{i-1}^{\bullet}(c)$ and $x^z\coloneqq \prod_{c\in z}x_c$. Then	
	\[ \sum_{\substack{\pi \in \LSQ(n)^{\bullet k} \\ \shift(\pi) = s \\ \dw(\pi) = z }} q^{\dinv(\pi)}t^{\area(\pi)} x^\pi 
	= t^{\maj(z)} q^{b(z,s)} \prod_{i=0}^\ell\left( \prod_{c \in \N} \qbinom{ w_{i,s}(c) + z_i(c) - 1}{z_i(c)}_q  q^{z_i^\bullet (c)\choose 2}\qbinom{w_{i,s}^\bullet (c)}{z_i^\bullet(c)}_q \right)x^z. \]
\end{theorem}

The proof of this result is similar to the one described by Haglund and Sergel in \cite[Theorem 3.2]{Haglund-Sergel-2019} for Dyck paths, except that we consider repeated labels. For the sake of completeness, we repeat some of their arguments here.

\begin{proof}
	Let us begin by noting that the right hand side of this equation consists of a finite number of terms different from $1$. Indeed $z_i(c) = z_i^\bullet(c) = 0$ for all but a finite number of elements of $\N$ and thus all but a finite number of $q$-binomials are equal to $1$, which means that the product is actually finite. 
	
	Next, observe that for any $\pi\in \LSQ(n)^{\bullet k}$ with $\dw(\pi)= z$ we trivially have $x^\pi=x^z$, so we only need to consider the $q,t$-enumerators. It is also not difficult to see that for any such path $\maj(z)=\area(\pi)$, indeed, 
	\begin{align*}
		\area(\pi) & = \ell \cdot \#\rho_{\ell} + (\ell-1)\cdot \#\rho_{\ell-1} + \cdots + 1 \cdot \#\rho_1 \\
		& = \rho_\ell + (\rho_\ell+\rho_{\ell-1}) + \cdots + (\rho_{\ell}+\rho_{\ell-1}+\cdots +\rho_1)= \maj(z).
	\end{align*} 
	
	For the dinv, we will construct all the paths of a given diagonal word and shift, starting from the empty path, all the while keeping track of the dinv. We do this by applying the procedure described below. For each step, we give illustrate by a (partial) construction of the paths with diagonal word $44223\!\!\stackrel{\bullet}{\raisebox{0 em}{3}}\stackrel{\bullet}{\raisebox{0 em}{3}}\stackrel{\bullet}{\raisebox{0 em}{0}}\!\!11\!\!\stackrel{\bullet}{\raisebox{0 em}{2}}$.
	
	First, we construct all the undecorated paths whose diagonal word is $\tilde \rho_s\dots \tilde \rho_0$. The decorations will be added in the last step.  
	
	\begin{enumerate}
		\item Starting from the empty path, we insert the numbers of $\tilde\rho_s$ into the main diagonal, starting with the biggest labels and draw the unique path whose vertical steps are labelled by these numbers. If $c$ is the biggest label in $\tilde\rho_s$ then $w_{s,s}(c)= 1$ and so $\qbinom{w_{s,s}+z_s(c)-1}{z_s(c)}_q=1$, indeed there is only one way to insert $z_s(c)$ labels equal to $c$ into the main diagonal and this creates 0 units of dinv because steps with the same label do not create dinv among each other. In general take $c$ a number of $\tilde \rho_s$. Then $w_{s,s}(c)$ is equal to the number of numbers in $\tilde\rho_s$ bigger than $c$, that have already be inserted, plus $1$ if $c \neq 0$. So when inserting $z_s(c)$ numbers equal to $c$ into the diagonal there are $\binom{w_{s,s}(c)-1+z_s(c)}{z_s(c)}$ ways to do it (even if $c$ is $0$, as the leftmost label on the main diagonal cannot be a $0$, which explains the absence of the extra $+1$) and the contribution to the dinv is counted by the $q$-analogue. Indeed we have to choose an interlacing between the $w_{s,s}(c)-1$ labels that are already there and the $z_s(c)$ labels equal to $c$ and each time a $c$ precedes a bigger label, one unit of dinv is created.
		
		\begin{figure}[!ht]
			\begin{center}
				\includegraphics{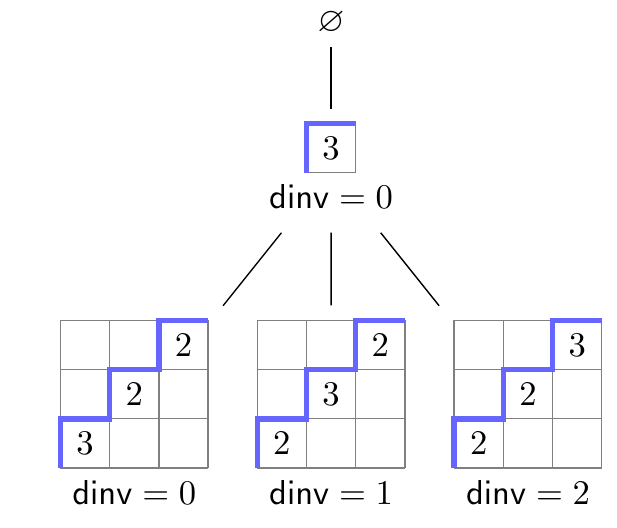}
			\end{center}
			\caption{Insertion of undecorated numbers into diagonal $y=x$}
		\end{figure}
		
	\item Next we add the numbers in $\tilde\rho_{s+1}, \tilde\rho_{s+2}, \dots, \tilde\rho_l$ (in that order) by adding the numbers in $\tilde \rho_i$ in the $i$-th diagonal and drawing the unique path with this labelling. As before, we insert the numbers from biggest to smallest in each run. When adding the $z_i(c)$ numbers equal to $c$ of $\tilde \rho_i$, we can add them on top of smaller numbers in $\tilde \rho_{i-1}$ or directly northeast of the numbers in $\tilde\rho_i$ that have already been inserted, i.e. that are bigger. There may also be consecutive $c$'s in the diagonal. So there are $w_{i,s}(c)$ labels after which the $c$'s may be inserted. So the insertion of the $z_i(c)$ labels equal to $c$ uniquely corresponds to an interlacing of the $z_i(c)$ $c$'s and $w_{i,s}(c)$ possible insertion positions, that does not start with a $c$ (indeed a $c$ must be inserted \emph{after} one of the $w_{i,s}(c)$ positions). Since each time one of the $c$'s precedes one of the $w_{i,s}(c)$ discussed labels, one unit of dinv is created (primary for bigger labels in $\tilde \rho_i$ and secondary for smaller labels in $\tilde \rho_{i-1}$), this dinv contribution is $q$-counted by $\qbinom{w_{s,i}(c)-1+z_i(c)}{z_i(c)}_q$.
	
	\begin{figure}[!ht]
		\begin{center}
			\includegraphics{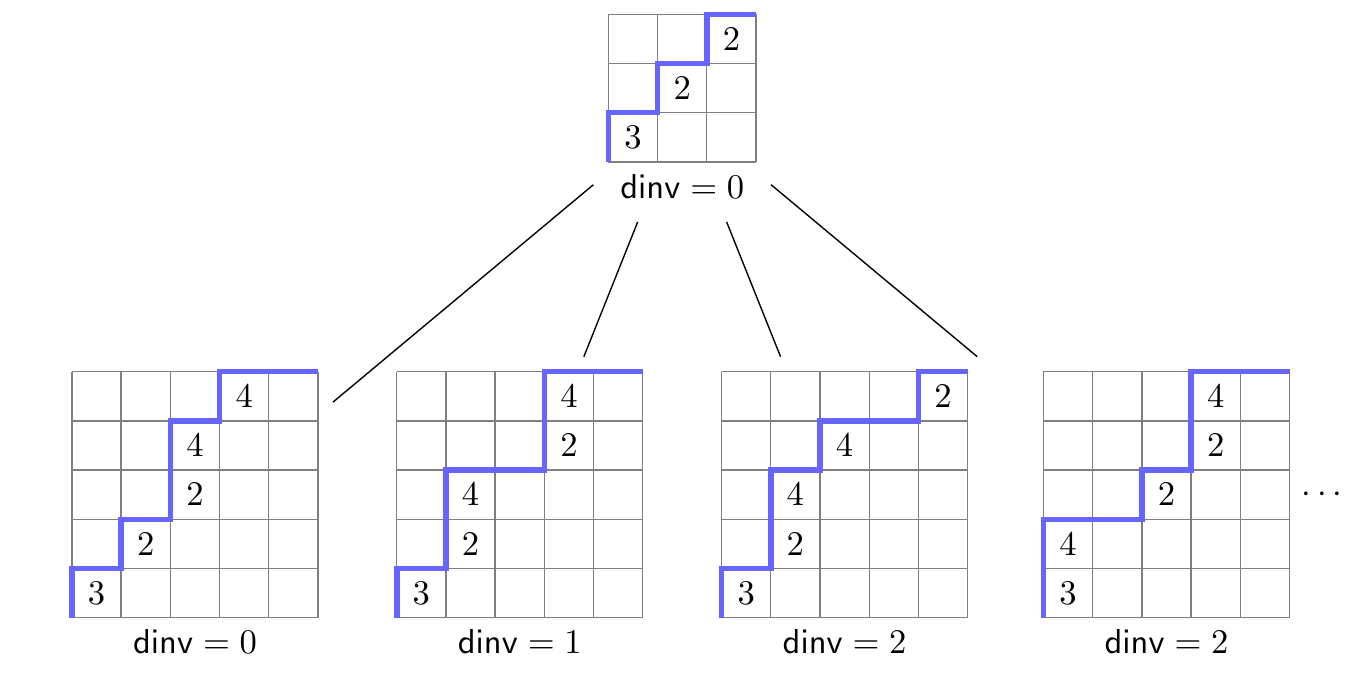}
		\end{center}
		\caption{Insertion of undecorated numbers into diagonal $y=x+1$}
	\end{figure}
	
	\item Next we add the numbers in runs $\tilde \rho_{s-1}, \tilde\rho_{s-2},\dots, \tilde\rho_0$ (in that order), this time from smallest to biggest numbers in each run. As before, we insert the numbers in run $\tilde\rho_i$ into the $i$-th diagonal. When adding the $z_i(c)$ numbers equal to $c$ in $\tilde\rho_i$, they can either be inserted directly underneath a bigger number from $\tilde \rho_{i+1}$ or directly southwest of a label of $\tilde \rho_i$ that has already been inserted, i.e. that is smaller. There may also be consecutive $c$'s in the diagonal. So there are again $w_{s,i}(c)$ places we may insert $c$. Note that the last $c$ must be underneath a bigger label from $\tilde \rho_{i+1}$, since the path must end with an east step. Thus choosing an interlacing between the $w_{i,s}(c)$ positions and the $z_i(c)$ $c$'s, ending with a label of the first kind, we get a unique insertion whose dinv is $q$-counted by $\qbinom{w_{s,i}(c)-1+z_i(c)}{z_i(c)}_qq^{\delta_{c>0} \sum_{i=0}^{s-1} \sum_{c \in \N} z_i(c)}$ since each time an insertion position precedes a $c$ a unit of dinv is created and each positive label underneath the main diagonal creates a unit of bonus dinv.
	
	\begin{figure}[!ht]	
		\begin{center}
			\includegraphics{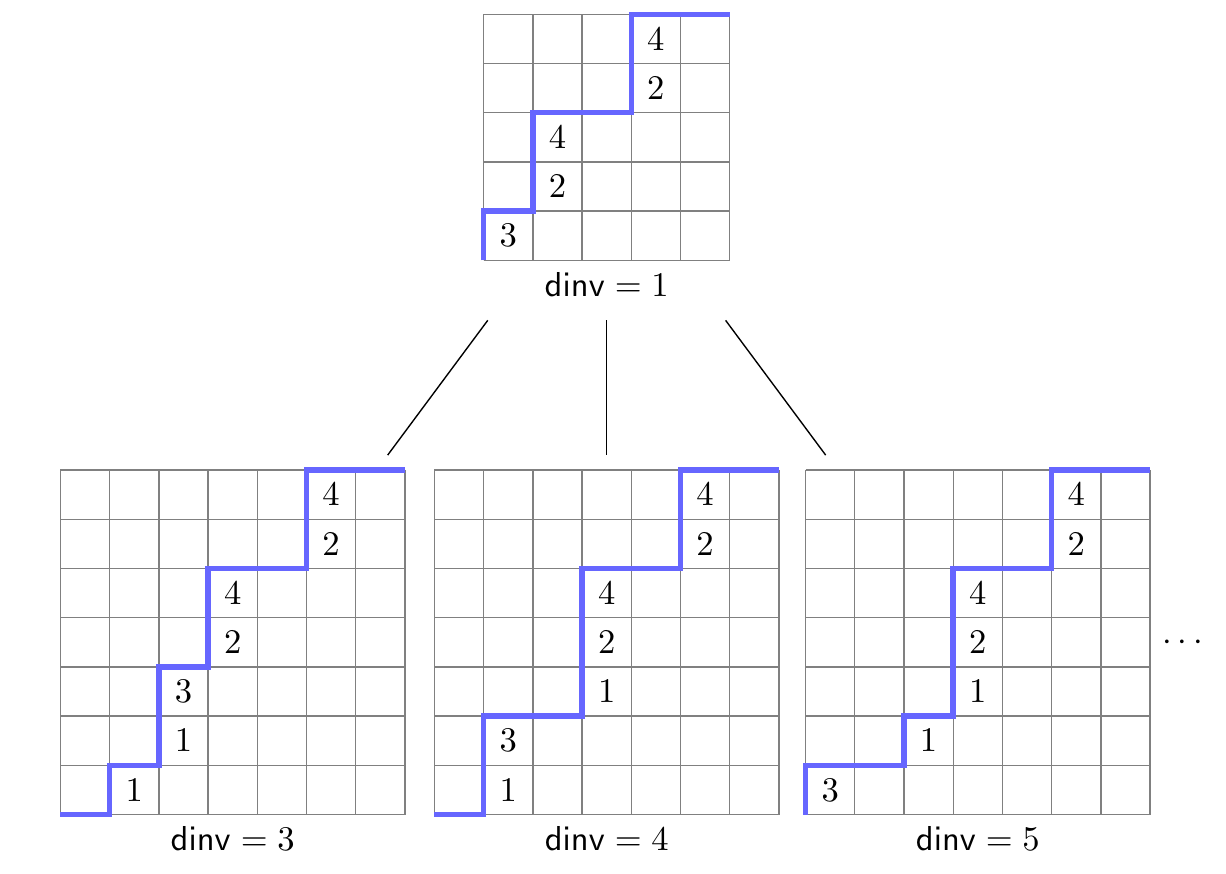}
		\end{center}
		\caption{Insertion of undecorated numbers into diagonal $y=x-1$}
	\end{figure}

	\item Finally we add the decorated labels. We will insert, for each $i$, $z_i^\bullet(c)$ decorated valleys labelled $c$ into the $i$-th diagonal.   Consider $S_{i,c}$, the set of size $w_{i,s}^\bullet(c)$, of all the steps that have already been inserted and would create dinv with a the decorated valley labelled $c$ inserted to its right in the $i$-th diagonal. In other words, the (undecorated) steps in the $i+1$-th diagonal with a label bigger than $c$ and the (undecorated) steps in the $i$-th diagonal with a label smaller than $c$. 
	
	The insertion is slightly different above and below the main diagonal. First, we consider the insertion above the main diagonal, take $i\in \{s,\dots, \ell\}$.
	Choose a subset $T\subseteq S_{i,c}$ of size $z_i^\bullet(c)$. We insert one decorated valley labelled $c$ in the $i$-th diagonal to the right of each element of $T$ and to the left of the next element of $S_{i,c}$. We claim that there is always a unique way to do this and that this yields all the possible paths (surjectivity).
		
		\textbf{Surjectivity and uniqueness.} There are two things to show. Firstly, we have to show that a decorated valley cannot be inserted to the left of all the elements is $S_{i,c}$. Indeed, we are inserting decorated valleys and decorations must be placed on contractible valleys. This means that if $(\pi, w, dv)$ is a valley-decorated labelled square path and $j\in dv$ then
		\begin{itemize}
			\item either $a_{j-1}=a_{j}$ and $w_{j-1}<w_{j}$, in which case $(j-1,j)$ is a primary inversion.
			\item or $a_{j-1}> a_j$ in which case there must be a $k<i$ such that $a_k= a_j$ and $a_{k+1}= a_k +1$ (in other words $k+1$ is a rise). Then either $w_{k+1}> w_j$, in which case $(k+1,j)$ is an inversion, or $w_{k+1}\leq w_j$, in which case $w_k < w_{k+1}\leq w_j$ so $(k,j)$ is an inversion. 
		\end{itemize}
		% TODO label and schedule number are both named w, this sucks.
		 It follows that when inserting a decorated valley at least one unit of dinv is created to its left, in other words, at least one element of $S_{i,c}$ is to its left.  
		
		Secondly, we must argue that there can never be two insertions in between two consecutive elements of $S_{i,c}$. In other words, there must always be an element of $S_{i,c}$ between two decorated valleys labelled $c$ in the $i$-th diagonal. Indeed 
		\begin{itemize}
			\item if one such valley is followed by a vertical step $s$, its label must be bigger then $c$ and lie in the $i+1$-th diagonal. Thus, $s$ is an element of $S_{i,c}$. 
			\item if one such valley is followed by a horizontal step the path hits the $i$-th diagonal at a point $p$. The next step cannot be another decorated valley or it would not be a \emph{contractible} valley.  Now we can apply a similar argument as above to the portion of the path starting from $p$ to deduce the existence of an element in $S_{i,c}$ that lies after $p$ and before the next occurence of the relevant decorated valleys.
		\end{itemize}
		This also implies that there exist no two distinct ways of inserting our decorated valley in between two elements of $S_{i,c}$. So the insertion must be unique.
		
		\textbf{Existence.} We now show that given an element $t$ of $S_{i,c}$, there always exists a way to insert a decorated valley of the discussed kind to its right and to the left of the the next element of $S_{i,c}$. It will follow that we can insert the valleys after elements of $T$ one by one and since each must lie strictly between elements of $S_{i,c}$ these insertions are independent. There some cases and subcases to consider. 
		
		\begin{itemize}
			\item [\textbf{Case 1.}]The step $t$ lies in the $i$-th diagonal, and thus is labelled with a number smaller than $c$, call it $S$.
			\begin{itemize}
				\item [\textbf{Case 1.1}] The step $t$ is followed by a vertical step.
					\begin{itemize}
						\item [\textbf{Case 1.1.1}] The label of the vertical step following $t$, is bigger than $c$. Call it $B$. Directly following $t$, insert a horizontal step followed by a decorated vertical step labelled $c$. Then continue the path with the vertical step labelled $B$, since $S<c<B$, the step labelled $c$ is a contractible valley and the condition on the columns of labelled paths is respected.
						
						\begin{figure}[H]
							\begin{center}
								\includegraphics{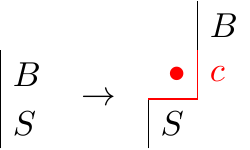}
							\end{center}							
						\end{figure}
					
						\item [\textbf{Case 1.1.2}] The label of the vertical step following $t$, is smaller than $c$. Call it $\tilde S$. Now consider the portion of the path between the endpoint of $t$ and the first point $p$ where the path crosses the $i+1$-th diagonal with two consecutive horizontal steps, one ending and one beginning at $p$ ($p$ always exists since the path must return to the main diagonal). If, in this portion of path, there is an occurence of a vertical step contained in the $i+1$-th diagonal and labelled with a number bigger then $c$, call it $B$, then this step is an element of $S_{i,c}$. Insert, directly before this step labelled $B$, a decorated vertical step labelled $c$ and before that step a horizontal step. Since $B$ is bigger than $c$, the labelling is valid. Furthermore, since the  discussed portion of the path stays weakly above the $i+1$-th diagonal, the step labelled $B$ must be preceded by a horizontal step and so the inserted valley is preceded by two horizontal steps and thus is contractible.
						 
						 Otherwise, the discussed portion of path does not contain an element of $S_{i,c}$. In that case, insert directly after the horizontal step ending at $p$ a horizontal step followed by a decorated vertical step labelled $c$ and continue the path with the horizontal step starting at $p$. Since the inserted decorated valley is preceded by two horizontal steps, it is always contractible.
						 
						 \begin{figure}[H]
						 	\begin{minipage}{.375 \textwidth}
						 		\centering
						 		\includegraphics{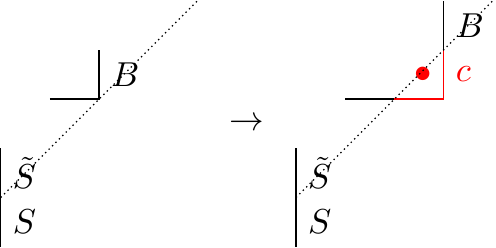}
						 	\end{minipage}%
						 	\begin{minipage}{.375 \textwidth}
						 		\centering
						 		\includegraphics{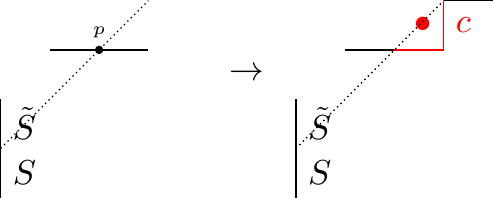}
						 	\end{minipage}
						 \end{figure}
					\end{itemize}
				
				\item [\textbf{Case 1.2}]  The step $t$ is followed by a horizontal step. In this case, we may simply insert, directly after $t$, a horizontal step followed by a decorated vertical step labelled $c$ and continue the path with the horizontal step that followed $t$. Since $S<c$, the valley we created is contractible.
				
				\begin{figure}[H]
					\begin{center}
						\includegraphics{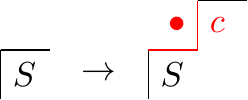}
					\end{center}
				\end{figure}
			\end{itemize}
		
			\item [\textbf{Case 2.}] The step $t$ lies in the $i+1$-th diagonal, and thus is labelled with a number bigger than $c$, call it $B$. The reasoning here is very similar to case 1.1.2.: after $t$ the path must cross the $i+1$-th diagonal. Depending on wether or not there is an element of $S_{i,c}$ in this portion of path, we have two types of insertions. We include a diagram of each situation.
			
			\begin{figure}[H]
				\begin{minipage}{.4 \textwidth}
					\centering
					\includegraphics{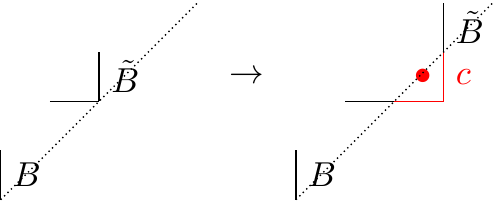}
				\end{minipage}%
				\begin{minipage}{.4 \textwidth}
					\centering
					\includegraphics{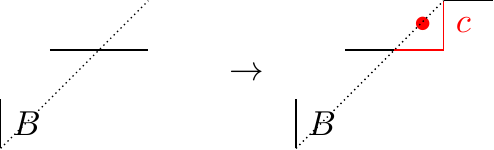}
				\end{minipage}
			\end{figure}				
		\end{itemize}
	
		It follows that the insertion of $z_i^\bullet(c)$ decorated valleys labelled $c$ into the $i$-th diagonal is equivalent to the choice of a subset $T\subseteq S_{i,c}$ of size $z_i^\bullet(c)$. As mentioned above, each such an insertion creates at least one unit of dinv, which can be thought of as being cancelled out by the subtraction of the number of decorated valleys in the computation of the dinv. Additionally, every pair of elements of $T$ creates one unit of dinv, as the decorated valley following the second one creates dinv with the first one. This explains the factor $q^{z_i^\bullet(c)\choose 2}$. Lastly one unit of dinv is created each time an element of $S_{i,c}\setminus T$ precedes an element of $T$ in the path, which explains $\qbinom{w_{i,s}^\bullet(c)}{z_i(c)}_q$
		
		Now for the insertion below the main diagonal, take $i\in \{0,\dots, s-1\}$. The argument is very similar, up to a few minor changes. Indeed, now it is no longer true that any decorated valley must create a unit of dinv with an element to its left. However, there must be an element of $S_{i,c}$ that is to the right of all the decorated valleys labelled $c$ in the $i$-th diagonal. Indeed let $t$ be such a decorated valley.
	
	\begin{itemize}
		\item if $t$ is followed by a vertical step, it must be contained in the $i+1$-th diagonal and its label must be bigger then $c$ so it must be an element of $S_{i,c}$; 
		\item if $t$ is followed by a horizontal step and so the path hits the $i$-th diagonal at a point $p$, which lies beneath the main diagonal. Since the path must end with an east step above the main diagonal we may deduce the existence of two consecutive vertical steps after $p$, the first one starting at the $i$-th diagonal. If the first of these steps has a label smaller that $c$, it is an element of $S_{i,c}$. If it is bigger of equal to $c$, the succeeding vertical step must have a label bigger than $c$ and thus must be in $S_{i,c}$. 
	\end{itemize}

	So given a subset $T\subseteq S_{i,c}$ of size $z_i^\bullet(c)$, we will insert a decorated valley in the $i$-th diagonal to the \emph{left} of each element in $T$ and to the right of all the elements of $S_{i,c}$ preceding the element it. 
	
	 Except for when $T$ contains the very first step of $S_{i,c}$, we can recycle the existence of insertion argument given above. Indeed, we can replace all the elements of $T$ with its predecessors in $S_{i,c}$, obtaining a set $\tilde T$ and apply the described insertions to the right. There is just one subtlety: in that discussion, we used the fact that if the path crosses the $i$-th diagonal vertically, it must proceed to cross it again, horizontally. This is not necessarily true if $i\in \{0,\dots, s-1\}$. However, as we have shown, there must be at least one element of $S_{i,c}$ to the right of all the elements of $\tilde T$ and this is sufficient to apply the same argument.	 
	 
	 The only thing that is now left to show that, for $i < s-1$, or $i = s-1$ and $c > 0$, it is (uniquely) possible to insert a decorated valley labelled $c$ into the $i$-th diagonal, to the left of all the elements in $S_{i,c}$. If $i = s-1$ and $c=0$, then this is not possible. Indeed by definition here cannot be a decorated valley labelled $0$ in the first row if the corresponding area letter is $-1$. Any other decorated valley labelled $0$ with area word $-1$ must be preceded by two horizontal steps (or the valley would not be \emph{contractible}) forcing a positive label on the main diagonal (which is an element of $S_{s-1,0}$) to appear on its left. This restriction explains the presence of the extra $-1$ in the definition of $w_{s-1,s}(0)$.
	 
	 If it exists, let us call $f$ the first vertical step of the existing path contained in the $i$-th diagonal ($f$ is not necessarily an element of $S_{i,c}$ as its label, call it $F$, might be bigger than $c$).
	 
	 \begin{itemize}
		\item[\textbf{Case 1.}] The step $f$ is an element of or occurs before every element in $S_{i,c}$. Then one of two situations occur. Either $f$ is preceded by a horizontal step, in which case that horizontal step must be preceded by another horizontal step, indeed if is were preceded by a vertical step, this would lie in the $i$-th diagonal and so $f$ would not be the first. In that case insert a horizontal step followed by a vertical decorated step labelled $c$ in between these two horizontal steps. Or, $f$ is preceded by a vertical step, so the begin point of $f$ is a point were the path crosses the $i$-th diagonal vertically. Since the path started at $(0,0)$, there must be a point before $f$ where the path crosses this diagonal horizontally, with two consecutive horizontal steps (notice that this point must be unique by definition of $f$). These two horizontal steps must be preceded by a third horizontal step because if not, this contradicts the definition of $f$. Insert a horizontal step followed by a decorated vertical step labelled $c$ directly after the first (from the left) of these three horizontal steps.
		
		\begin{figure}[H]
			\begin{minipage}{.38 \textwidth}
				\centering
				\includegraphics{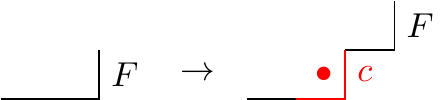}
			\end{minipage}%
			\begin{minipage}{.48 \textwidth}
				\centering
				\includegraphics{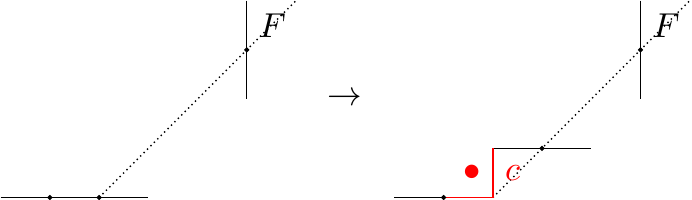}
			\end{minipage}
		\end{figure}

		\item[\textbf{Case 2.}] The leftmost element of $S_{i,c}$ is to the left of all (if any exist) the vertical steps in the $i$-th diagonal, in which case it must be a step $t$ in the $i+1$-th diagonal labelled with a number bigger than $c$, call it $B$. It follows that $t$ must be preceded by a horizontal step. Insert a horizontal step followed by a decorated vertical step labelled $c$ directly after this horizontal step.
		
		\begin{figure}[H]
			\begin{center}
				\includegraphics{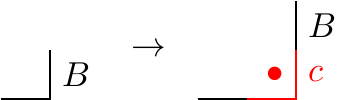}
			\end{center}
		\end{figure}		
	 \end{itemize}
	 Essentially the same argument as before ensures that there cannot be two decorated valleys inserted between two elements of $S_{i,c}	$ (or before all of its elements). This ensures the unicity of the insertion given a choice of $T$. Now any two inserted valleys must create one unit of dinv explaining and any element of $T$ succeeding an element of $S_{i,c}\setminus T$ creates a unit of dinv. Each inserted decorated non-zero valley below the diagonal creates a unit of bonus dinv by definition, which we can think of as getting cancelled out with the subtraction of the number of decorated non-zero valleys in the computation of the dinv. Furthermore, for each decorated zero valley below the diagonal, there must be a corresponding non-decorated non-zero valley, which also contributed one unit to the dinv that we can think of as cancelled out with the subtraction of the number of decorated zero valleys. 
	 The contribution to the dinv of this procedure is thus $q$-counted by $q^{z_i^\bullet (c)\choose 2}\qbinom{w_{i,s}^\bullet (c)}{z_i^\bullet(c)}_q$. 
	
	\begin{figure}[!ht]
		\begin{center}
			\includegraphics[scale=0.8]{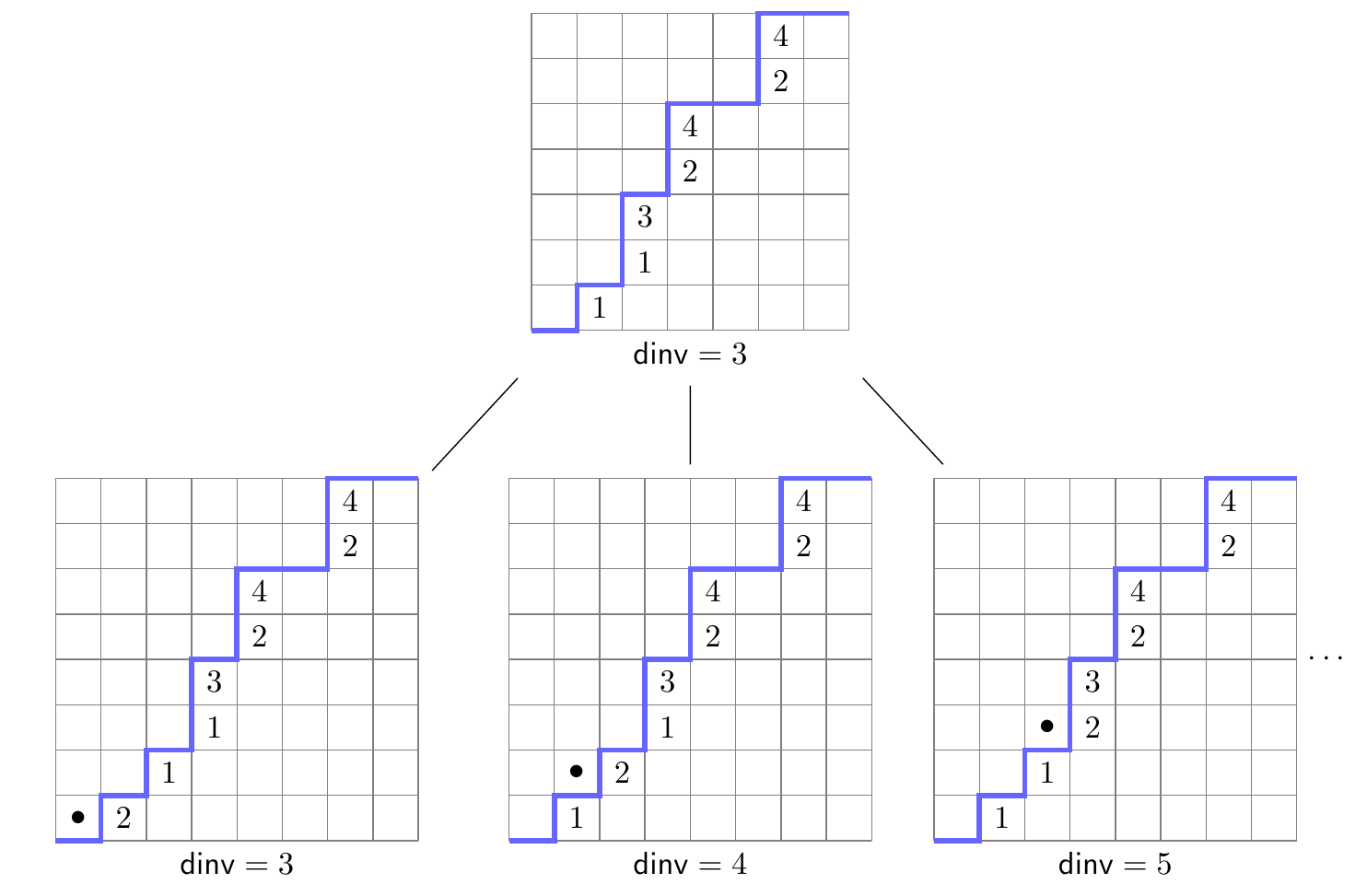}
		\end{center}
		\caption{Insertion of $\stackrel{\bullet}{\raisebox{0 em}{2}}$ into diagonal $y=x-1$}
	\end{figure}

	\begin{figure}[!ht]
		\begin{center}
			\includegraphics[scale=0.8]{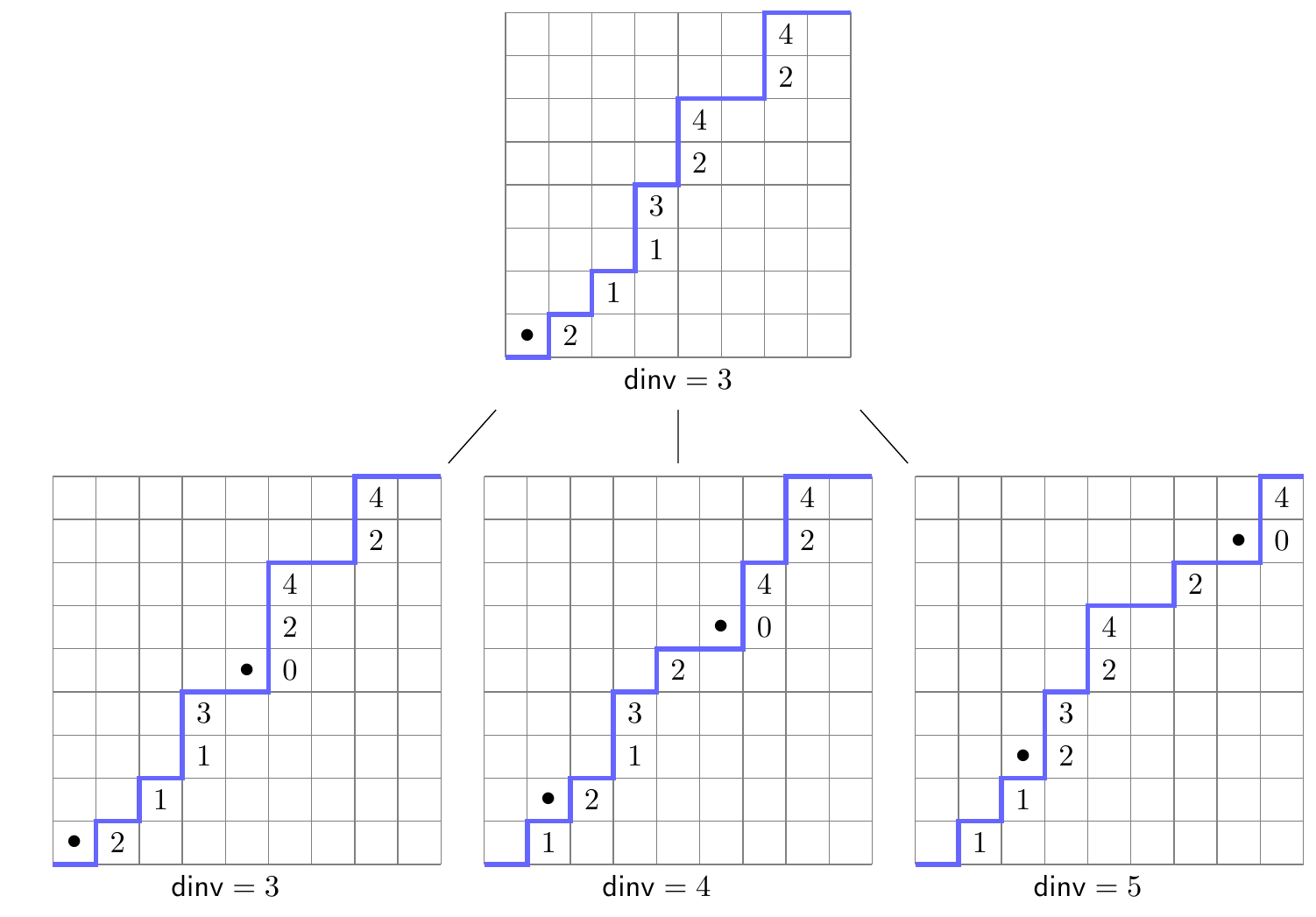}
		\end{center}
		\caption{Insertion of $\stackrel{\bullet}{\raisebox{0 em}{0}}$ into diagonal $y=x-1$}
	\end{figure}
	
	\begin{figure}[!ht]
		\begin{center}
			\includegraphics[scale=0.8]{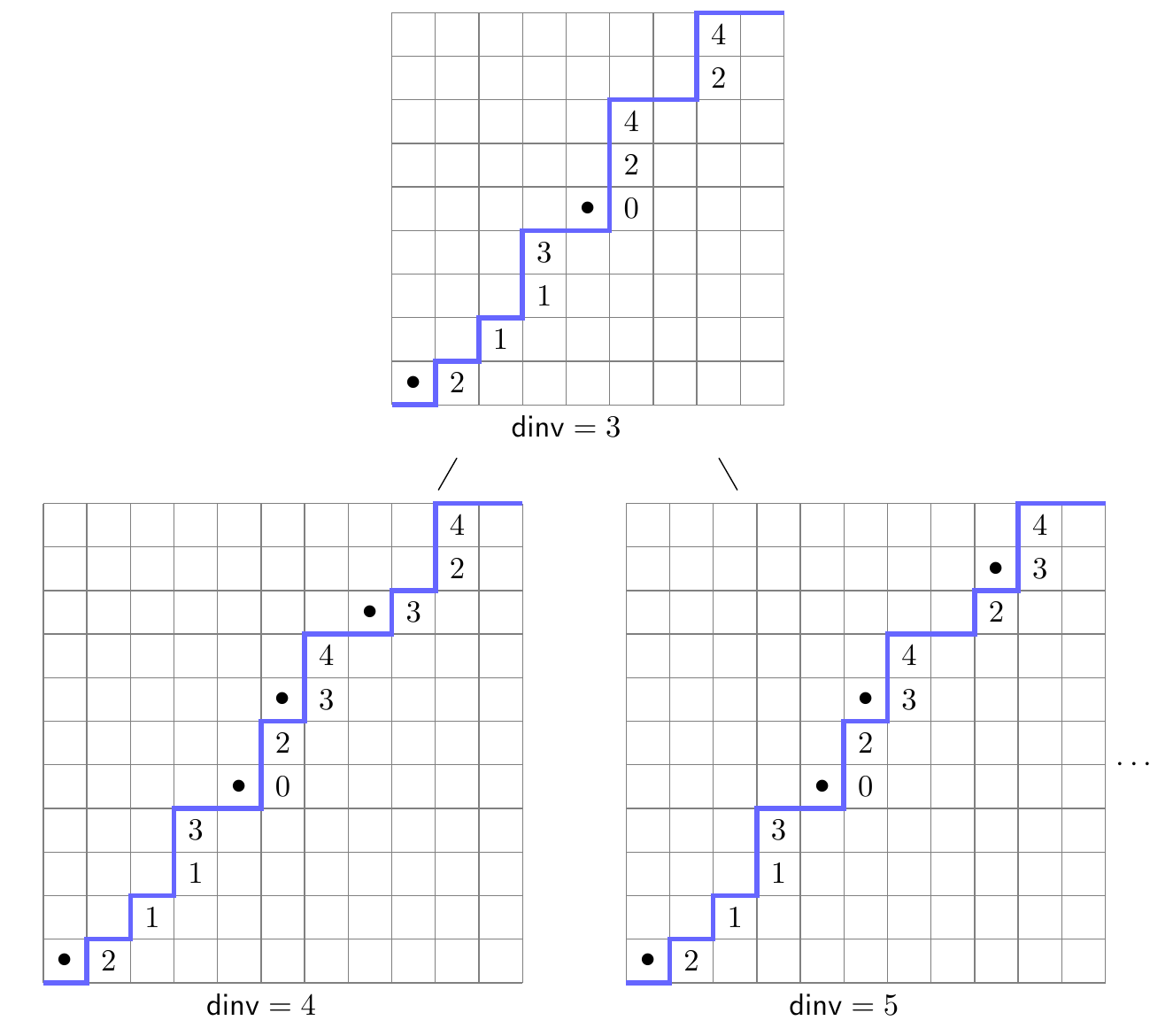}
		\end{center}
		\caption{Insertion of two $\stackrel{\bullet}{\raisebox{0 em}{3}}$'s into diagonal $y=x-1$}
	\end{figure}

	\end{enumerate}
\end{proof}

\section{The valley Delta implies the valley square}
%\begin{definition}
%	Let $z \coloneqq \dw(\pi, w, dv)$ be the diagonal word of a valley-decorated labelled square path $(\pi, w, dv)$ such that $z = \rho_\ell \cdots \rho_0$, where the $\rho_i$'s are its runs. We define the corresponding \emph{reduced word} $z'$ to be the word composed of the unmarked element of $z$. Similarly, we define $\rho'_i$ to be the word composed of the unmarked elements of $\rho_i$, and its \emph{$i$-th run multiplicity function}
%	\begin{align*}
%		z'_i \colon \N_0 & \rightarrow \N \\
%		x & \mapsto \# \{ x \in \rho'_i \}
%	\end{align*}%
%\end{definition}
%
%\begin{remark}
%	If $z$ is the diagonal word of a valley-decorated labelled square path, then the $\rho'_i$'s are the runs of $z'$. This is not immediately obvious from the definition. % TODO Prove.
%\end{remark}

Let us define \[ \LSQ_{q,t;x}(z,s) \coloneqq \sum_{\substack{\pi \in \LSQ(n)^{\bullet k} \\ \shift(\pi) = s \\ \dw(\pi) = z}} q^{\dinv(\pi)} t^{\area(\pi)} x^\pi, \]
where $z$ is the diagonal word of a path in $\LSQ(n)^{\bullet k}$. We have the following.

\begin{theorem}
	\label{thm:shift-by-1}
	\[ \LSQ_{q,t;x}(z,s) = q^{\# \rho'_{s-1}-z^\bullet_{s-2}(0)} \frac{[\# \rho'_{s}-z_{s-1}^\bullet(0)]_q}{[\# \rho'_{s-1}z_{s-1}^\bullet(0)]_q} \LSQ_{q,t;x}(z,s-1). \]
\end{theorem}

\begin{proof}
	By Theorem~\ref{thm:factorisation}, we have
	
	\[ \LSQ_{q,t;x}(z,s) =  t^{\maj(z)} q^{b(z,s)} \prod_{i=0}^\ell\left( \prod_{c \in \N} \qbinom{ w_{i,s}(c) + z_i(c) - 1}{z_i(c)}_q  q^{z_i^\bullet (c)\choose 2}\qbinom{w_{i,s}^\bullet (c)}{z_i^\bullet(c)}_q \right)x^z. \]
	
	and
	
	\[ \LSQ_{q,t;x}(z,s-1) =  t^{\maj(z)} q^{b(z,s-1)} \prod_{i=0}^\ell\left( \prod_{c \in \N} \qbinom{ w_{i,s-1}(c) + z_i(c) - 1}{z_i(c)}_q  q^{z_i^\bullet (c)\choose 2}\qbinom{w_{i,s-1}^\bullet (c)}{z_i^\bullet(c)}_q \right)x^z. \]
	
	Now, by definition we have $b(z,s) - b(z, s-1) = \# \rho'_{s-1}-z^\bullet_{s-2}(0)$. Also, for $c\neq 0$ we have that $w_{i,s}^\bullet(c)$ does not depend on $s$ and $w_{i,s}^\bullet(0) = w_{i,s-1}^\bullet(0)$ for $i\not\in \{s-2,s-1\}$. Furthermore for $c \in \N$ we have that $w_{i,s}(c) = w_{i,s-1}(c)$ for $i \not \in \{s-1, s\}$. By simplifying all these terms, we get
	
	\begin{align*} 
	\LSQ_{q,t;x}(z,s) =  q^{\# \rho'_{s-1}-z_{s-2}^\bullet(0)} &\prod_{c \in \N} 
	\frac{
	\qbinom{ w_{s,s}(c) + z_s(c) - 1}{z_s(c)}_q 
	\qbinom{ w_{s-1,s}(c) + z_{s-1}(c) - 1}{z_{s-1}(c)}_q 
	}
	{
	\qbinom{ w_{s,s-1}(c) + z_s(c) - 1}{z_s(c)}_q 
	\qbinom{ w_{s-1,s-1}(c) + z_{s-1}(c) - 1}{z_{s-1}(c)}_q
		} 
	\\
	& \times \prod_{c\in \N}
	\frac{
	\qbinom{w_{s-1,s}^\bullet(0)}{z_{s-1}^\bullet(0)}_q
	\qbinom{w_{s-2,s}^\bullet(0)}{z_{s-2}^\bullet(0)}_q
	}
	{
	\qbinom{w_{s-1,s-1}^\bullet(0)}{z_{s-1}^\bullet(0)}_q
	\qbinom{w_{s-2,s-1}^\bullet(0)}{z_{s-2}^\bullet(0)}_q
	}
	\LSQ_{q,t;x}(z,s-1) 
	\end{align*}
	
	so all that's left to do is to compute the $q$-binomials and check that the product yields the desired result.
	
	Recall that
	\begin{align*}
		w_{s,s}(c) & = 1 - \delta_{c,0} + \sum_{a > c} z_s(a), \qquad && w_{s-1,s}(c) = \sum_{a > c} z_s(a) + \sum_{a < c} z_{s-1}(a), \\
		w_{s-1,s-1}(c) & = 1 - \delta_{c,0} + \sum_{a > c} z_{s-1}(a) \qquad && w_{s,s-1}(c) = \sum_{a > c} z_s(a) + \sum_{a < c} z_{s-1}(a).
	\end{align*} and that 
	\begin{align*}
		w_{s-1,s}^\bullet(0)  = \#\rho'_s-1 && w_{s-2,s}^\bullet(0) = \rho'_{s-1} &&
		w_{s-1,s-1}^\bullet(0)  = \#\rho'_s && w_{s-2,s-1}^\bullet(0) = \#\rho'_{s-1}-1
	\end{align*}
	
	Let $m = \max \{ c \in \N \mid z_s(c) > 0 \text{ or } z_{s-1}(c) > 0 \}$. We have
	\begin{align*}
		&\prod_{c \in \N} \frac{\qbinom{ w_{s,s}(c) + z_s(c) - 1}{z_s(c)}_q}{\qbinom{ w_{s-1,s-1}(c) + z_{s-1}(c) - 1}{z_{s-1}(c)}_q}  = \prod_{c=1}^{m} \frac{[z_{s-1}(c)]_q!}{[z_s(c)]_q!} \cdot \frac{[w_{s,s}(c) + z_s(c) - 1]_q!}{[w_{s,s}(c) - 1]_q!} \cdot \frac{[w_{s-1,s-1}(c) - 1]_q!}{[w_{s-1,s-1}(c) + z_{s-1}(c) - 1]_q!} \\
		& = \prod_{c=0}^{m} \frac{[z_{s-1}(c)]_q!}{[z_s(c)]_q!} \cdot \frac{[\sum_{a \geq c} z_s(a) - \delta_{c,0}]_q!}{[\sum_{a > c} z_s(a) - \delta_{c,0}]_q!} \cdot \frac{[\sum_{a > c} z_{s-1}(a)-\delta_{c,0}]_q!}{[\sum_{a \geq c} z_{s-1}(a) - \delta_{c,0}]_q!} \\
		& = \frac{[\sum_{a \geq 0} z_s(a)]_q!}{[\sum_{a > 0} z_s(a)]_q!} \cdot \frac{[\sum_{a > 0} z_{s-1}(a)]_q!}{[\sum_{a \geq 0} z_{s-1}(a)]_q!}\prod_{c=0}^{m} \frac{[z_{s-1}(c)]_q!}{[z_s(c)]_q!}  \prod_{c=1}^{m} \frac{[\sum_{a \geq c} z_s(a)]_q!}{[\sum_{a > c} z_s(a)]_q!} \cdot \frac{[\sum_{a > c} z_{s-1}(a)]_q!}{[\sum_{a \geq c} z_{s-1}(a)]_q!} \\
		& = \frac{[\#\tilde\rho_s-1]_q! [\#\rho'_{s-1}-1]_q!}{[\# \rho'_s-1]_q![\# \tilde \rho_{s-1}-1]_q!} \frac{[\sum_{a \geq 1} z_s(a)]_q!}{[\sum_{a \geq 1} z_{s-1}(a)]_q!} \cdot \prod_{c=0}^{m} \frac{[z_{s-1}(c)]_q!}{[z_s(c)]_q!} \prod_{c=1}^{m} \frac{[\sum_{a > c} z_s(a)]_q!}{[\sum_{a > c} z_s(a)]_q!} \cdot \frac{[\sum_{a > c} z_{s-1}(a)]_q!}{[\sum_{a > c} z_{s-1}(a)]_q!} \\
		& = \frac{[\#\tilde\rho_s-1]_q! [\#\rho'_{s-1}-1]_q!}{[\# \rho'_s-1]_q![\# \tilde \rho_{s-1}-1]_q!} \frac{[\# \rho'_s]_q!}{[\# \rho'_{s-1}]_q!} \cdot \prod_{c=0}^{m} \frac{[z_{s-1}(c)]_q!}{[z_s(c)]_q!} \\
		& = \frac{[\#\tilde\rho_s-1]_q!}{[\# \tilde \rho_{s-1}-1]_q!} \frac{[\# \rho'_s]_q}{[\# \rho'_{s-1}]_q} \cdot \prod_{c=0}^{m} \frac{[z_{s-1}(c)]_q!}{[z_s(c)]_q!}
	\end{align*}%
	and
	\begin{align*}
		\prod_{c \in \N} \frac{\qbinom{ w_{s-1,s}(c) + z_{s-1}(c) - 1}{z_{s-1}(c)}_q}{\qbinom{ w_{s,s-1}(c) + z_s(c) - 1}{z_s(c)}_q} & = \prod_{c=0}^{m} \frac{[z_s(c)]_q!}{[z_{s-1}(c)]_q!} \cdot \frac{[w_{s-1,s}(c) + z_{s-1}(c) - 1]_q!}{[w_{s-1,s}(c) - 1]_q!} \cdot \frac{[w_{s,s-1}(c) - 1]_q!}{[w_{s,s-1}(c) + z_s(c) - 1]_q!} \\
		& = \prod_{c=0}^{m} \frac{[z_s(c)]_q!}{[z_{s-1}(c)]_q!} \cdot \frac{[w_{s-1,s}(c) + z_{s-1}(c) - 1]_q!}{[w_{s,s-1}(c) + z_s(c) - 1]_q!} \\
%		& = \prod_{c=1}^{m} \frac{[z_s(c)]_q!}{[z_{s-1}(c)]_q!} \cdot \frac{[\sum_{a > c} z_s(a) + \sum_{a \leq c} z_{s-1}(a) - 1]_q!}{[\sum_{a > c} z_s(a) + \sum_{a < c} z_{s-1}(a) - 1]_q!} \cdot \frac{[\sum_{a > c} z_s(a) + \sum_{a < c} z_{s-1}(a) - 1]_q!}{[\sum_{a \geq c} z_s(a) + \sum_{a < c} z_{s-1}(a) - 1]_q!} \\
		& = \prod_{c=0}^{m} \frac{[z_s(c)]_q!}{[z_{s-1}(c)]_q!} \prod_{c=0}^{m} \frac{[\sum_{a > c} z_s(a) + \sum_{a \leq c} z_{s-1}(a) - 1]_q!}{[\sum_{a \geq c} z_s(a) + \sum_{a < c} z_{s-1}(a) - 1]_q!} \\
		& = \frac{\prod_{c=0}^{m} [\sum_{a > c} z_s(a) + \sum_{a \leq c} z_{s-1}(a) - 1]_q!}{\prod_{c=0}^{m} [\sum_{a > c-1} z_s(a) + \sum_{a \leq c-1} z_{s-1}(a) - 1]_q!} \cdot \prod_{c=0}^{m} \frac{[z_s(c)]_q!}{[z_{s-1}(c)]_q!} \\
		& = \frac{\prod_{c=0}^{m} [\sum_{a > c} z_s(a) + \sum_{a \leq c} z_{s-1}(a) - 1]_q!}{\prod_{c=-1}^{m-1} [\sum_{a > c} z_s(a) + \sum_{a \leq c} z_{s-1}(a) - 1]_q!} \cdot \prod_{c=0}^{m} \frac{[z_s(c)]_q!}{[z_{s-1}(c)]_q!} \\
		& = \frac{[\sum_{a > m} z_s(a) + \sum_{a \leq m} z_{s-1}(a) - 1]_q!}{[\sum_{a > -1} z_s(a) + \sum_{a \leq -1} z_{s-1}(a) - 1]_q!} \cdot \prod_{c=0}^{m} \frac{[z_s(c)]_q!}{[z_{s-1}(c)]_q!} \\
		& = \frac{[\sum_{a \leq m} z_{s-1}(a) - 1]_q!}{[\sum_{a \geq 0} z_s(a) - 1]_q!} \cdot \prod_{c=0}^{m} \frac{[z_s(c)]_q!}{[z_{s-1}(c)]_q!} \\
		& = \frac{[\# \tilde \rho_{s-1} - 1]_q!}{[\# \tilde \rho'_s - 1]_q!} \cdot \prod_{c=0}^{m} \frac{[z_s(c)]_q!}{[z_{s-1}(c)]_q!}
	\end{align*}%
	and
	\begin{align*}
		\frac{\qbinom{w_{s-1,s}^\bullet(0)}{z_{s-1}^\bullet(0)}_q \qbinom{w_{s-2,s}^\bullet(0)}{z_{s-2}^\bullet(0)}_q}{\qbinom{w_{s-1,s-1}^\bullet(0)}{z_{s-1}^\bullet(0)}_q\qbinom{w_{s-2,s-1}^\bullet(0)}{z_{s-2}^\bullet(0)}_q} & = \frac{[w_{s-1,s}^\bullet(0)]_q![w_{s-2,s}^\bullet(0)]_q![w_{s-1,s-1}^\bullet(0)-z_{s-1}^\bullet(0)]_q![w_{s-2,s-1}^\bullet(0)-z_{s-2}^\bullet(0)]_q!}{[w_{s-1,s-1}^\bullet(0)]_q![w_{s-2,s-1}^\bullet(0)]_q![w_{s-1,s}^\bullet(0)-z_{s-1}^\bullet(0)]_q![w_{s-2,s}^\bullet(0)-z_{s-2}^\bullet(0)]_q!} 
		\\&= \frac{[\#\rho'_{s}-1]_q![\#\rho'_{s-1}]_q!}{[\#\rho'_{s}]_q![\#\rho'_{s-1}-1]_q!}
		\frac{[\#\rho'_{s}-z_{s-1}^\bullet(0)]_q![\#\rho'_{s-1}-1-z_{s-2}^\bullet(0)]_q!}{[\#\rho'_{s}-1-z_{s-1}^\bullet(0)]_q![\#\rho'_{s-1}-z_{s-2}^\bullet(0)]_q!}\\
		&=\frac{[\#\rho'_{s-1}]_q}{[\#\rho'_{s}]_q}
		\frac{[\#\rho'_{s}-z_{s-1}^\bullet(0)]_q}{[\#\rho'_{s-1}-z_{s-2}^\bullet(0)]_q}.
	\end{align*} 
	
	Taking the product we get, after obvious cancellations
	\begin{align*}
		 %&\left(\frac{[\#\tilde\rho_s-1]_q!}{[\# \tilde \rho_{s-1}-1]_q!} \frac{[\# \rho'_s]_q}{[\# \rho'_{s-1}]_q} \cdot \prod_{c=0}^{m} \frac{[z_{s-1}(c)]_q!}{[z_s(c)]_q!}\right) \left(\frac{[\# \tilde \rho_{s-1} - 1]_q!}{[\# \tilde \rho'_s - 1]_q!} \cdot \prod_{c=0}^{m} \frac{[z_s(c)]_q!}{[z_{s-1}(c)]_q!}\right)\left(\frac{[\#\rho'_{s-1}]_q}{[\#\rho'_{s}]_q} \frac{[\#\rho'_{s}-z_{s-1}^\bullet(0)]_q}{[\#\rho'_{s-1}-z_{s-2}^\bullet(0)]_q}\right) \\
		\frac{[\#\rho'_{s}-z_{s-1}^\bullet(0)]_q}{[\#\rho'_{s-1}-z_{s-2}^\bullet(0)]_q}
	\end{align*}
	which is exactly what we wanted to show.
\end{proof}

\begin{corollary}
	If $\# \rho'_0 \neq 0$, then
	\[ \LSQ_{q,t;x}(z,s) = q^{b(z,s)} \frac{[\# \rho'_s - z_{s-1}^\bullet(0)]_q}{[\# \rho'_0]_q} \LD_{q,t;x}(z). \]
\end{corollary}

\begin{proof}
	It follows immediately by applying \ref{thm:shift-by-1} $s$ times and recalling that $z_{-1}^\bullet(0) = 0$.
\end{proof}

\begin{corollary}
	\label{cor:square-to-dyck}
	\[ \LSQ'_{q,t;x}(n \backslash r)^{\bullet k} = \frac{[n-k]_q}{[r]_q} \LD_{q,t;x}(n \backslash r)^{\bullet k} \]
\end{corollary}

\begin{proof}
	Given a marked word $z$ with $\ell$ runs and $\rho'_0 \neq 0$, we have
	\begin{align*}
		\sum_{s=0}^{\ell} \LSQ_{q,t;x}(z,s) & = \sum_{s=0}^{\ell} q^{b(z,s)} \frac{[\# \rho'_s  - z_{s-1}^\bullet(0)]_q}{[\# \rho'_0]_q} \LD_{q,t;x}(z) \\
		& = \frac{\sum_{s=0}^{\ell} q^{b(z,s)} [\# \rho'_s - z_{s-1}^\bullet(0)]_q}{[\# \rho'_0]_q} \LD_{q,t;x}(z) \\
		& = \frac{[\sum_{s=0}^{\ell} (\# \rho'_s - z_{s-1}^\bullet(0))]_q}{[\# \rho'_0]_q} \LD_{q,t;x}(z)
	\end{align*}
	and now taking the sum over all the marked words $z$ of length $n$ with $k$ decorations and $\rho'_0 = r$, since for any such $z$, $\sum_{s=0}^{\ell} \# \rho'_s = n - k + \sum_{s=0}^{\ell} z_{s-1}^\bullet(0)$ (the total number of non-decorated positive labels), the thesis follows immediately.
\end{proof}

\begin{theorem}[Conditional modified Delta square conjecture, valley version]
	If Conjecture~\ref{conj:gen-valley-delta-touching} holds, then so does Conjecture~\ref{conj:gen-valley-square-2}. As a special case, if Conjecture~\ref{conj:valley-delta-touching} holds, then so does Conjecture~\ref{conj:valley-square-2}.
\end{theorem}

\begin{proof}
	We recall the statement of Conjecture~\ref{conj:gen-valley-delta-touching}, which is
	\[ \Delta_{h_m} \Theta_{e_k} \nabla E_{n-k, r} = \sum_{\pi \in \LD(m, n \backslash r)^{\bullet k}} q^{\dinv(\pi)} t^{\area(\pi)} x^\pi. \]
	Applying Corollary~\ref{cor:square-to-dyck}, we have
	\[ \frac{[n-k]_q}{[r]_q} \Delta_{h_m} \Theta_{e_k} \nabla E_{n-k, r} = \sum_{\pi \in \LSQ'(m, n \backslash r)^{\bullet k}} q^{\dinv(\pi)} t^{\area(\pi)} x^\pi. \]
	Summing over $r$ and using Proposition~\ref{prop:pn_Enk}, we get
	\[ \Delta_{h_m} \Theta_{e_k} \nabla \omega(p_{n-k}) = \sum_{\pi \in \LSQ'(m,n)^{\bullet k}} q^{\dinv(\pi)} t^{\area(\pi)} x^\pi, \]
	as desired.
\end{proof}
\section{Concluding remarks}\label{sec:concluding}
As we mentioned before, the slightly contrived conditions on the positions of the steps labelled with zeros can be reformulated quite naturally by considering a step labelled $0$ as the ``pushing'' of a step labelled $\infty$. 

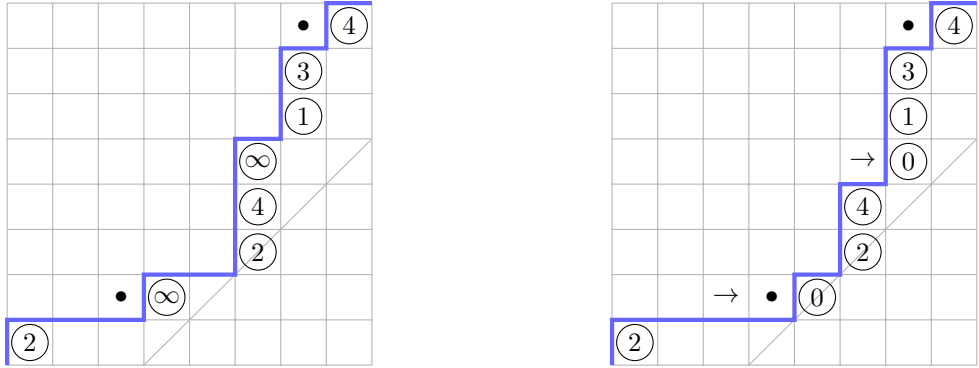
\begin{figure}[!ht]
  \begin{minipage}{0.5\textwidth}
    \centering
    \begin{tikzpicture}[scale = 0.6]
      \draw[step=1.0, gray!60, thin] (0,0) grid (8,8);
      \draw[gray!60, thin] (3,0) -- (8,5);
      %    \draw[gray!60, thin] (8,6) -- (9,6) -- (9,8) (8,7) -- (10,7) -- (10,8) (8,8) -- (11,8);
      %    \draw[gray!60, thin] (8,5) -- (11,8);
      
      \draw[blue!60, line width=1.6pt] (0,0) -- (0,1) -- (1,1) -- (2,1) -- (3,1) -- (3,2) -- (4,2) -- (5,2) -- (5,3) -- (5,4) -- (5,5) -- (6,5) -- (6,6) -- (6,7) -- (7,7) -- (7,8) -- (8,8);
      
      \node at (2.5,1.5) {$\bullet$};
      \node at (6.5,7.5) {$\bullet$};
      
      \node at (0.5,0.5) {$2$};
      \draw (0.5,0.5) circle (.4cm); 
      \node at (3.5,1.5) {$\infty$};
      \draw (3.5,1.5) circle (.4cm); 
      \node at (5.5,2.5) {$2$};
      \draw (5.5,2.5) circle (.4cm); 
      \node at (5.5,3.5) {$4$};
      \draw (5.5,3.5) circle (.4cm); 
      \node at (5.5,4.5) {$\infty$};
      \draw (5.5,4.5) circle (.4cm); 
      \node at (6.5,5.5) {$1$};
      \draw (6.5,5.5) circle (.4cm); 
      \node at (6.5,6.5) {$3$};
      \draw (6.5,6.5) circle (.4cm); 
      \node at (7.5,7.5) {$4$};
      \draw (7.5,7.5) circle (.4cm);
      \end{tikzpicture}
  \end{minipage}%
  \begin{minipage}{0.5\textwidth}
    \centering
    \begin{tikzpicture}[scale = 0.6]
      \draw[step=1.0, gray!60, thin] (0,0) grid (8,8);
      \draw[gray!60, thin] (3,0) -- (8,5);
      %    \draw[gray!60, thin] (8,6) -- (9,6) -- (9,8) (8,7) -- (10,7) -- (10,8) (8,8) -- (11,8);
      %    \draw[gray!60, thin] (8,5) -- (11,8);
      
      \draw[blue!60, line width=1.6pt] (0,0) -- (0,1) -- (1,1) -- (2,1) -- (3,1) -- (4,1) -- (4,2) -- (5,2) -- (5,3) -- (5,4) -- (6,4) -- (6,5) -- (6,6) -- (6,7) -- (7,7) -- (7,8) -- (8,8);
      
      \node at (3.5,1.5) {$\bullet$};
      \node at (6.5,7.5) {$\bullet$};
            
      \node at (2.5,1.5) {$\rightarrow$};
      \node at (5.5,4.5) {$\rightarrow$};
      
      \node at (0.5,0.5) {$2$};
      \draw (0.5,0.5) circle (.4cm); 
      \node at (4.5,1.5) {$0$};
      \draw (4.5,1.5) circle (.4cm); 
      \node at (5.5,2.5) {$2$};
      \draw (5.5,2.5) circle (.4cm); 
      \node at (5.5,3.5) {$4$};
      \draw (5.5,3.5) circle (.4cm); 
      \node at (6.5,4.5) {$0$};
      \draw (6.5,4.5) circle (.4cm); 
      \node at (6.5,5.5) {$1$};
      \draw (6.5,5.5) circle (.4cm); 
      \node at (6.5,6.5) {$3$};
      \draw (6.5,6.5) circle (.4cm); 
      \node at (7.5,7.5) {$4$};
      \draw (7.5,7.5) circle (.4cm);
    \end{tikzpicture}
  \end{minipage}
  \caption{``Pushing'' of $\infty$'s.}
  \label{fig:pushing-algorithm}
\end{figure}
Performing this manoeuvre does not change the dinv (if we define that the $\infty$'s under the main diagonal do not contribute to the bonus dinv and that there are no $\infty$'s on the base diagonal). The area changes by a constant factor equal to the number of zeros. 

Several open problems arise from our discussion. There is no interpretation of the symmetric function $\Delta_{h_m} \Theta_{e_k} \nabla \omega(p_{n-k})$ in terms of rise-decorated square paths, for which also the schedule formula is lacking. This is one of the very few instances where the valley version seems to be easier to treat than the rise version. Understanding the rise version better might lead to a unified valley-rise conjecture interpreting $\Theta_{e_j} \Theta_{e_k} \nabla e_{n-k-j}$. 

Lastly, it would be nice to show that the valley Delta conjecture implies the generalised valley Delta conjecture. Given that, our results would be conditional only on the valley Delta conjecture. There might be a way to prove this using the ``pushing'' manoeuvre described above to interpret the behaviour of the $h_j^\perp$ operator. We have some symmetric function identities suggesting that this avenue might be fruitful, and some of these conjectural identities are strongly suggested by certain relations among the combinatorial objects.

% Bibliography

\bibliographystyle{amsalpha}
\bibliography{bibliography}

\begin{bibdiv}
\begin{biblist}

\bib{Bergeron-Garsia-ScienceFiction-1999}{incollection}{
      author={Bergeron, Fran\c{c}ois},
      author={Garsia, Adriano~M.},
       title={Science fiction and {M}acdonald's polynomials},
        date={1999},
   booktitle={Algebraic methods and {$q$}-special functions (montr\'eal, {QC},
  1996)},
      series={CRM Proc. Lecture Notes},
      volume={22},
   publisher={Amer. Math. Soc., Providence, RI},
       pages={1\ndash 52},
      review={\MR{1726826}},
}

\bib{Bergeron-Garsia-Haiman-Tesler-Positivity-1999}{article}{
      author={Bergeron, Fran\c{c}ois},
      author={Garsia, Adriano~M.},
      author={Haiman, Mark},
      author={Tesler, Glenn},
       title={Identities and positivity conjectures for some remarkable
  operators in the theory of symmetric functions},
        date={1999},
        ISSN={1073-2772},
     journal={Methods Appl. Anal.},
      volume={6},
      number={3},
       pages={363\ndash 420},
         url={https://doi.org/10.4310/MAA.1999.v6.n3.a7},
        note={Dedicated to Richard A. Askey on the occasion of his 65th
  birthday, Part III},
      review={\MR{1803316}},
}

\bib{Can-Loehr-2006}{article}{
      author={Can, Mahir},
      author={Loehr, Nicholas},
       title={A proof of the {$q,t$}-square conjecture},
        date={2006},
        ISSN={0097-3165},
     journal={J. Combin. Theory Ser. A},
      volume={113},
      number={7},
       pages={1419\ndash 1434},
         url={https://doi.org/10.1016/j.jcta.2006.01.002},
      review={\MR{2259069}},
}

\bib{Carlsson-Mellit-ShuffleConj-2018}{article}{
      author={Carlsson, Erik},
      author={Mellit, Anton},
       title={A proof of the shuffle conjecture},
        date={2018},
        ISSN={0894-0347},
     journal={J. Amer. Math. Soc.},
      volume={31},
      number={3},
       pages={661\ndash 697},
         url={https://doi.org/10.1090/jams/893},
      review={\MR{3787405}},
}

\bib{DAdderio-Iraci-VandenWyngaerd-TheBible-2019}{article}{
      author={D'Adderio, Michele},
      author={Iraci, Alessandro},
      author={Vanden~Wyngaerd, Anna},
       title={Decorated {D}yck paths, polyominoes, and the {D}elta conjecture},
        date={2018},
     journal={Mem. Amer. Math. Soc.},
}

\bib{DAdderio-Iraci-VandenWyngaerd-DeltaSquare-2019}{article}{
      author={D'Adderio, Michele},
      author={Iraci, Alessandro},
      author={Vanden~Wyngaerd, Anna},
       title={{The Delta Square Conjecture}},
        date={201903},
        ISSN={1073-7928},
     journal={International Mathematics Research Notices},
      eprint={rnz057.pdf},
         url={https://doi.org/10.1093/imrn/rnz057},
}

\bib{DAdderio-Iraci-VandenWyngaerd-Delta-t0-2018}{article}{
      author={D'Adderio, Michele},
      author={Iraci, Alessandro},
      author={Vanden~Wyngaerd, Anna},
       title={{The generalized Delta conjecture at t=0}},
        date={2019Jan},
     journal={arXiv e-prints},
       pages={arXiv:1901.02788},
      eprint={1901.02788},
         url={https://arxiv.org/abs/1901.02788},
}

\bib{DAdderio-Iraci-VandenWyngaerd-GenDeltaSchroeder-2019}{article}{
      author={D'Adderio, Michele},
      author={Iraci, Alessandro},
      author={Vanden~Wyngaerd, Anna},
       title={{The Schröder case of the generalized Delta conjecture}},
        date={2019},
        ISSN={0195-6698},
     journal={European Journal of Combinatorics},
      volume={81},
       pages={58 \ndash  83},
  url={http://www.sciencedirect.com/science/article/pii/S0195669819300447},
}

\bib{DAdderio-Iraci-VandenWyngaerd-Theta-2019}{article}{
      author={D'Adderio, Michele},
      author={Iraci, Alessandro},
      author={Vanden~Wyngaerd, Anna},
       title={Theta operators, refined {D}elta conjectures, and coinvariants},
        date={2019Jun},
     journal={arXiv e-prints},
      eprint={1906.02623},
         url={https://arxiv.org/abs/1906.02623},
}

\bib{Garsia-Haglund-Remmel-Yoo-2019}{article}{
      author={Garsia, Adriano},
      author={Haglund, Jim},
      author={Remmel, Jeffrey~B.},
      author={Yoo, Meesue},
       title={A {P}roof of the {D}elta {C}onjecture {W}hen {$q=0$}},
        date={2019Jun},
        ISSN={0219-3094},
     journal={Annals of Combinatorics},
      volume={23},
      number={2},
       pages={317\ndash 333},
         url={https://doi.org/10.1007/s00026-019-00426-x},
}

\bib{Haglund-Book-2008}{book}{
      author={Haglund, James},
       title={The {$q$},{$t$}-{C}atalan numbers and the space of diagonal
  harmonics},
      series={University Lecture Series},
   publisher={American Mathematical Society, Providence, RI},
        date={2008},
      volume={41},
        ISBN={978-0-8218-4411-3; 0-8218-4411-3},
        note={With an appendix on the combinatorics of Macdonald polynomials},
      review={\MR{2371044}},
}

\bib{HHLRU-2005}{article}{
      author={Haglund, James},
      author={Haiman, Mark},
      author={Loehr, Nicholas},
      author={Remmel, Jeffrey~B.},
      author={Ulyanov, Anatoly},
       title={A combinatorial formula for the character of the diagonal
  coinvariants},
        date={2005},
        ISSN={0012-7094},
     journal={Duke Math. J.},
      volume={126},
      number={2},
       pages={195\ndash 232},
         url={https://doi.org/10.1215/S0012-7094-04-12621-1},
      review={\MR{2115257}},
}

\bib{Haglund-Remmel-Wilson-2018}{article}{
      author={Haglund, James},
      author={Remmel, Jeffrey~B.},
      author={Wilson, Andrew~T.},
       title={The {D}elta {C}onjecture},
        date={2018},
        ISSN={0002-9947},
     journal={Trans. Amer. Math. Soc.},
      volume={370},
      number={6},
       pages={4029\ndash 4057},
         url={https://doi.org/10.1090/tran/7096},
      review={\MR{3811519}},
}

\bib{Haglund-Rhoades-Shimozono-Advances}{article}{
      author={Haglund, James},
      author={Rhoades, Brendon},
      author={Shimozono, Mark},
       title={Ordered set partitions, generalized coinvariant algebras, and the
  {D}elta {C}onjecture},
        date={2018},
        ISSN={0001-8708},
     journal={Adv. Math.},
      volume={329},
       pages={851\ndash 915},
         url={https://doi.org/10.1016/j.aim.2018.01.028},
      review={\MR{3783430}},
}

\bib{Haglund-Sergel-2019}{article}{
      author={Haglund, James},
      author={Sergel, Emily},
       title={Schedules and the {D}elta {C}onjecture},
        date={2019Aug},
     journal={arXiv e-prints},
      eprint={1908.04732},
         url={https://arxiv.org/pdf/1908.04732.pdf},
}

\bib{Haiman-nfactorial-2001}{article}{
      author={Haiman, Mark},
       title={Hilbert schemes, polygraphs and the {M}acdonald positivity
  conjecture},
        date={2001},
        ISSN={0894-0347},
     journal={J. Amer. Math. Soc.},
      volume={14},
      number={4},
       pages={941\ndash 1006},
         url={https://doi.org/10.1090/S0894-0347-01-00373-3},
      review={\MR{1839919}},
}

\bib{Haiman-Vanishing-2002}{article}{
      author={Haiman, Mark},
       title={Vanishing theorems and character formulas for the {H}ilbert
  scheme of points in the plane},
        date={2002},
        ISSN={0020-9910},
     journal={Invent. Math.},
      volume={149},
      number={2},
       pages={371\ndash 407},
         url={https://doi.org/10.1007/s002220200219},
      review={\MR{1918676}},
}

\bib{Loehr-Warrington-square-2007}{article}{
      author={Loehr, Nicholas~A.},
      author={Warrington, Gregory~S.},
       title={Square {$q,t$}-lattice paths and {$\nabla(p_n)$}},
        date={2007},
        ISSN={0002-9947},
     journal={Trans. Amer. Math. Soc.},
      volume={359},
      number={2},
       pages={649\ndash 669},
         url={https://doi.org/10.1090/S0002-9947-06-04044-X},
      review={\MR{2255191}},
}

\bib{Macdonald-Book-1995}{book}{
      author={Macdonald, Ian~G.},
       title={Symmetric functions and {H}all polynomials},
     edition={Second},
      series={Oxford Mathematical Monographs},
   publisher={The Clarendon Press, Oxford University Press, New York},
        date={1995},
        ISBN={0-19-853489-2},
        note={With contributions by A. Zelevinsky, Oxford Science
  Publications},
      review={\MR{1354144}},
}

\bib{Qiu-Wilson-2019}{misc}{
      author={Qiu, Dun},
      author={Wilson, Andrew~Timothy},
       title={The valley version of the {E}xtended {D}elta {C}onjecture},
        date={2019},
}

\bib{Remmel-Wilson-2015}{article}{
      author={Remmel, Jeffrey~B.},
      author={Wilson, Andrew~T.},
       title={An extension of {M}ac{M}ahon's equidistribution theorem to
  ordered set partitions},
        date={2015},
        ISSN={0097-3165},
     journal={J. Combin. Theory Ser. A},
      volume={134},
       pages={242\ndash 277},
         url={https://doi.org/10.1016/j.jcta.2015.03.012},
      review={\MR{3345306}},
}

\bib{Rhoades-2018}{article}{
      author={Rhoades, Brendon},
       title={Ordered set partition statistics and the {D}elta {C}onjecture},
        date={2018},
        ISSN={0097-3165},
     journal={J. Combin. Theory Ser. A},
      volume={154},
       pages={172\ndash 217},
         url={https://doi.org/10.1016/j.jcta.2017.08.017},
      review={\MR{3718065}},
}

\bib{Leven-2016}{article}{
      author={Sergel, Emily},
       title={A proof of the {S}quare {P}aths {C}onjecture},
        date={2017},
        ISSN={0097-3165},
     journal={J. Combin. Theory Ser. A},
      volume={152},
       pages={363\ndash 379},
         url={https://doi.org/10.1016/j.jcta.2017.06.013},
      review={\MR{3682738}},
}

\bib{Wilson-Equidistribution}{article}{
      author={Wilson, Andrew~T.},
       title={An extension of {M}ac{M}ahon's equidistribution theorem to
  ordered multiset partitions},
        date={2016},
        ISSN={1077-8926},
     journal={Electron. J. Combin.},
      volume={23},
      number={1},
       pages={Paper 1.5, 21},
  url={https://www.combinatorics.org/ojs/index.php/eljc/article/view/v23i1p5/pdf},
      review={\MR{3484710}},
}

\bib{Zabrocki-4Catalan-2016}{article}{
      author={Zabrocki, Mike},
       title={A proof of the {$4$}-variable {C}atalan polynomial of the {D}elta
  conjecture},
        date={2016-09},
     journal={ArXiv e-prints},
      eprint={1609.03497},
         url={https://arxiv.org/abs/1609.03497},
}

\bib{Zabrocki-Delta-Module-2019}{article}{
      author={{Zabrocki}, Mike},
       title={A module for the {D}elta conjecture},
        date={2019Feb},
     journal={arXiv e-prints},
       pages={arXiv:1902.08966},
      eprint={1902.08966},
         url={https://arxiv.org/abs/1902.08966},
}

\end{biblist}
\end{bibdiv}

\end{document}